\documentclass[12pt]{amsart}

\usepackage[top=1in, bottom=1in, left=1in, right=1in]{geometry}

\usepackage{color}
\usepackage{amssymb}
\usepackage{comment}
\usepackage{pdfpages}
\usepackage{graphicx}
\usepackage{dcolumn}

\RequirePackage[numbers]{natbib}
\RequirePackage[colorlinks=true, pdfstartview=FitV, linkcolor=blue,
  citecolor=blue, urlcolor=blue]{hyperref}
\RequirePackage{hypernat}
\usepackage{paralist}

\newcommand{\cK}{{\mathcal {K}}}
\newcommand{\cL}{{\mathcal {L}}}
\newcommand{\cC}{{\mathcal {C}}}

\usepackage{graphicx}
\graphicspath{{./figures/}}
\usepackage{amsfonts}
\usepackage{amsmath}
\usepackage{amsthm}
\usepackage{amssymb}
\usepackage{amsbsy}
\usepackage{epsfig}
\usepackage{fullpage}
\usepackage{natbib, mathrsfs} 
\usepackage{verbatim}
\usepackage[latin1]{inputenc}
\usepackage{mhequ}

\numberwithin{equation}{section}

\renewcommand{\nabla}{D}
\newcommand{\opnm}[3]{\|#1\|_{{\cal L}({\cal H}^{#2},{\cal H}^{#3})}}
\newcommand{\C}{\mathcal{C}}
\newcommand{\cM}{{\mathcal M}}
\newcommand{\cH}{{\mathcal H}}
\newcommand{\Hel}{\cH^{s}}
\newcommand{\h}{\mathcal{H}}
\newcommand{\Hf}{\mathsf{H}}

\newcommand{\cHHl}{\cH^{s} \times \cH^{s}}

\newcommand{\bbR}{\mathbb R}

\newcommand{\bbW}{\mathbb W}
\newcommand{\sR}{\mathsf R}
\newcommand{\cal}[1]{\mathcal{#1}} 

\newcommand{\tr}{\mathrm{Trace}}

\newcommand{\RR}{\mathbb{R}}

\newcommand{\eqdef}{\overset{{\mbox{\tiny  def}}}{=}}

\newcommand{\bra}[1]{\langle #1 \rangle}

\newcommand{\be}{\begin{equs}}
\newcommand{\ee}{\end{equs}}

\newcommand{\Normal}{\textrm{N}}

\newcommand{\dist}{\overset{\mathcal{D}}{\sim}}

\newcommand{\EE}{\mathbb{E}}

\renewcommand\phi{\varphi}

\newcommand{\kd}{^{k,\delta}}
\newcommand{\kpod}{^{k+1,\delta}}
\newcommand{\de}{^{\delta}}
\newcommand{\lv}{\left\vert}
\newcommand{\rv}{\right\vert}
\newcommand{\nor}[1]{\|#1\|_s}
\newcommand{\nors}[1]{\|#1\|_{s\times s}}
\newcommand{\les}{\lesssim}

\def \be{\begin{equs}}
\def \ee{\end{equs}}
\def \sol{{SOL-HMC}\,}

\definecolor{darkred}{rgb}{.7,0,0}

\definecolor{darkgreen}{rgb}{0,0.7,0}

\definecolor{darkblue}{rgb}{0,0,0.7}

\newtheorem{theorem}{Theorem}[section]
\newtheorem{lemma}[theorem]{Lemma}
\newtheorem{remark}[theorem]{Remark}
\newtheorem{remarks}[theorem]{Remarks}

\newtheorem{prop}[theorem]{Proposition}
\newtheorem{assumptions}[theorem]{Assumptions}

\newtheorem{cond}[theorem]{Condition}

\begin{document}

\title[A Function Space HMC Algorithm With Second Order Langevin Diffusion Limit]
{A Function Space HMC Algorithm With Second Order Langevin Diffusion Limit}


\author{Michela Ottobre$^{\ast}$} 
\thanks{$^{\ast}$ michelaottobre@gmail.com, Imperial College London, 
    Mathematics Department, 
180 Queen's Gate, SW7 2AZ, London
}
\author{Natesh S. Pillai$^{\ddag}$}
\thanks{$^{\ddag}$pillai@fas.harvard.edu, 
   Department of Statistics
    Harvard University,1 Oxford Street, Cambridge
    MA 02138, USA}

\author{Frank J. Pinski$^{\sharp}$}
\thanks{$^{\sharp}$pinskifj@ucmail.uc.edu, Department of Physics,
University of Cincinnati,
PO Box 210011,
Cincinnati, OH 45221-0377, USA}

  \author{Andrew M. Stuart$^{\natural}$}
\thanks{$^{\natural}$a.m.stuart@warwick.ac.uk, 
Mathematics Institute,
    Warwick University, CV4 7AL, UK
}
  \thanks{NSP is partially supported by NSF-DMS 1107070}
\thanks{AMS is supported by EPSRC and ERC}


\maketitle
\begin{abstract}
We describe a new MCMC method optimized for the sampling of
probability measures on Hilbert space which have a density
with respect to a Gaussian; such measures arise in the
Bayesian approach to inverse problems, and in conditioned
diffusions. Our algorithm is based on two key design principles: 
(i) algorithms which 
are well-defined in infinite dimensions result in methods which do not 
suffer from the curse of dimensionality when they are
applied to approximations of the infinite dimensional
target measure on $\bbR^N$; 
(ii) non-reversible algorithms can have better mixing properties
compared to their reversible counterparts.
The method we introduce is based on the hybrid Monte Carlo
algorithm, tailored to incorporate these two design principles.
The main result of this paper states that the new algorithm, 
appropriately rescaled, converges weakly to a second order Langevin 
diffusion on Hilbert space; as a consequence the algorithm explores
the approximate target measures on $\bbR^N$ in a number of
steps which is independent of $N$.  We also present the underlying 
theory for the limiting non-reversible diffusion on Hilbert space, including
characterization of the invariant measure, and we describe
numerical simulations  demonstrating that the proposed
method has favourable mixing properties as
an MCMC algorithm.
\end{abstract}







\section{Introduction} Markov chain Monte Carlo (MCMC) algorithms for sampling from high dimensional probability distributions constitute an important part of 
Bayesian statistical inference. This paper is focussed
on the design and analysis of such algorithms to sample a probability distribution 
on an infinite dimensional Hilbert space 
$\cH$ defined via a density with respect to a Gaussian; such problems
arise in the Bayesian approach to inverse problems (or Bayesian
nonparametrics) \cite{Stua:10} and in the theory of conditioned diffusion
processes \cite{Hair:Stua:Voss:10}. 
Metropolis-Hastings algorithms \cite{Hast:70}
constitute a popular class of MCMC methods for sampling an arbitrary
probability measure. They proceed by constructing
an irreducible, {\em reversible} Markov chain by first proposing a candidate move
and then accepting it with a certain probability. The acceptance probability is chosen so as to preserve the detailed balance condition ensuring reversibility. In this work, we build on the generalized  Hybrid Monte Carlo (HMC) method
of \cite{H91} to construct a new {\em non-reversible}
MCMC method appropriate for sampling measures defined via density with respect
to a Gaussian measure on a Hilbert space. We also demonstrate that, for
a particular set of parameter values in the algorithm, there is
a  natural diffusion limit
to the second order Langevin (SOL) equation with invariant
measure given by the target. We thus name the new method 
the \sol algorithm. Our construction is motivated by the following two key design principles:
 \begin{enumerate}
 \item designing proposals which are well-defined on the 
Hilbert space results in MCMC
methods which do not suffer from the curse of dimensionality when
applied to sequences of approximating finite dimensional measures
on $\bbR^N$;
 \item non-reversible MCMC algorithms, which are hence not from the
Metropolis-Hastings class, can have better sampling properties 
in comparison with their reversible counterparts. 
 \end{enumerate}
The idea behind the first principle is explained in \cite{David}
which surveys a range of algorithms designed specifically
to sample measures defined via a density with respect to a Gaussian;
the unifying theme is that the proposal is reversible with 
respect to the underlying Gaussian so that the accept-reject
mechanism depends only on the  likelihood function and not the
prior distribution. The second principle above is also well-documented: 
non-reversible Markov chains, often constructed by 
performing individual time-reversible \footnote{For a definition of time-reversibility see Section \ref{sec:algorithm}.} steps successively
\cite{hwang1993,hwang2005}, or by building on Hamiltonian mechanics  
\cite{H91, diac:etal:2000,neal2010mcmc}), may have better 
mixing properties. 
 
Since the target distribution has support on an infinite dimensional space, 
practical implementation of MCMC involves  
discretizing  the parameter space, resulting
in a target measure on $\mathbb{R}^N$, with $N \gg 1$. 
It is well known that such discretization schemes
can suffer from the curse of dimensionality: the efficiency of 
the algorithm decreases as the dimension $N$ of the discretized 
space grows large. One way of understanding this is through
diffusion limits of the algorithm. In the context of measures
defined via density with respect to Gaussian this approach
is taken in the papers \cite{Matt:Pill:Stu:11, Pill:Stu:Thi:12} which show
that the random walk Metropolis and Langevin algorithms
require ${\mathcal O}(N)$ and ${\mathcal O}(N^{\frac13})$ steps
respectively to sample the approximating target measure
in $\bbR^N$. If, however, the algorithm is defined on Hilbert
space then it is possible to explore the target in ${\mathcal O}(1)$
steps and this may also be demonstrated by means of
a diffusion limit. The paper \cite{PST13} uses this
idea to study  a Metropolis-Hastings algorithm which is defined
on Hilbert space and is a small modification of the random walk Metropolis
method; the diffusion limit is a first order reversible
Langevin diffusion.  Moreover the diffusion limits  
in \cite{Matt:Pill:Stu:11,{Pill:Stu:Thi:12}} are derived under stationarity whereas
the results in \cite{Pill:Stu:Thi:12} hold for \emph{any} initial condition.
  The above discussion has important practical consequences: as implied
  by the above diffusion limits, algorithms which are well defined on
  the function spaces show an order of magnitude improvement in the mixing time
  in these high dimensional sampling problems. 
  
  Here we employ similar techniques as that of \cite{PST13} to study our new 
non-reversible MCMC method, and show that, after appropriate rescaling,
it converges to a second
order non-reversible Langevin diffusion. 
Our new algorithm is inspired
by similar algorithms in finite dimensions, starting with the
work of \cite{H91}, who showed how the momentum updates
could be correlated in the original HMC method of \cite{Duane1987216}, 
and the more recent work
\cite{bou2011patch} which made the explicit connection to
second order Langevin diffusions; a helpful overview and discussion
may be found in  \cite{neal2010mcmc}.
Diffusion limit results similar to ours
are proved in \cite{bou2011patch,Bou} for finite dimensional problems. 
In those papers
an accept-reject mechanism is appended to various standard
integrators for the first and second order Langevin equations,
and shown not to destroy the strong pathwise convergence
of the underlying methods. The reason for this is that rejections
are rare when small time-steps are used. The same reasoning
underlies the results we present here, although we consider
an infinite dimensional setting and use only weak convergence.
Another existing work underpinning that presented here
is the paper \cite{BPSS11} which generalizes the hybrid Monte Carlo
method for measures defined via density with respect to a Gaussian
so that it applies on Hilbert space. Indeed the algorithm
we introduce in this paper includes the one from \cite{BPSS11} as
a special case and uses the split-step (non-Verlet) integrator
first used there. The key idea of the splitting employed is to
split according to linear and nonlinear dynamics within the
numerical Hamiltonian integration step of the algorithm, rather
than according to position and momentum. This allows for an algorithm
which exactly preserves the underlying Gaussian reference measure,
without rejections, and is key to the fact that the methods are
defined on Hilbert space even in the the non-Gaussian case.

We now define the class of models to 
which our main results are applicable. 
Let $\pi_0$ and $\pi$ be two measures on a Hilbert space 
$\Bigl(\h,\langle \cdot,\cdot\rangle,\|\cdot\|\Bigr)$ and
assume that $\pi_0$ is Gaussian so that $\pi_0 = \mathrm{N}(0,C)$, 
with $C$ a covariance operator. The target measure $\pi$ is assumed
to be absolutely continuous with respect to $\pi_0$ and given by the
identity 
\begin{align}
\frac{d\pi}{d\pi_0}(x) = M_{\Psi} \exp \bigl( -\Psi(x) \bigr), \quad x \in \h
\label{eqn:targmeas0}
\end{align}
for a real valued functional $\Psi$ (which denotes the 
negative log-likelihood in the case of Bayesian inference) and $M_{\Psi}$ a 
normalizing constant.  Although the above formulation may
appear quite abstract, we emphasize that this points to
the wide-ranging applicability of our theory: the setting 
encompasses a large class of models arising in practice,
including nonparametric regression using Gaussian 
random fields and statistical inference for diffusion processes and bridge sampling
\cite{Hair:Stua:Voss:10, Stua:10}.

In Section \ref{sec:Hflow} we introduce 
our new algorithm. We start in a finite dimensional context
and then explain parametric choices made with reference
to the high or infinite dimensional setting. 
We demonstrate that various other algorithms defined on Hilbert space,
such as the function space MALA \cite{Besk:etal:08}
and function space HMC algorithms \cite{BPSS11}, are special cases.
In Section \ref{sec:2}
we describe the infinite dimensional setting in full and,
in particular, detail the relationship between the change of measure,
encapsulated in $\Psi$, and the properties of the Gaussian prior $\pi_0.$ 
Section \ref{sec:infinite} contains the theory of the SPDE which
both motivates our class of algorithms, and acts as a limiting process
for a specific instance of our algorithm applied on a 
sequence of spaces of increasing dimension $N$. We prove
existence and uniqueness of solutions to the SPDE and
characterize its invariant measure.
Section \ref{sec:difflim} contains statement of the
key diffusion limit Theorem \ref{thm:difflim}. Whilst the
structure of the proof is outlined in some detail, various
technical estimates are left for Appendices A and B.
Section \ref{sec:num} contains some numerics illustrating the
new algorithm in the context of a problem from the theory of
conditioned diffusions. We make some brief concluding remarks in
Section \ref{sec:conc}.

The new algorithm proposed and analyzed in this paper is of interest
for two primary reasons. Firstly, it contains a number of existing
function space algorithms as special cases and hence plays a useful
conceptual role in unifying these methods. Secondly numerical
evidence demonstrates that the method is comparable in efficiency
to the function space HMC method introduced in \cite{BPSS11}
for a test problem arising in conditioned diffusions; until now, the 
function space HMC method was the clear best choice as 
demonstrated numerically in \cite{BPSS11}. Furthermore, our numerical results indicate that for certain parameter choices in the \sol algorithm,
and for certain target measures, we are able to improve upon the
performance of the function space HMC algorithm, 
corroborating a similar observation made in
\cite{H91} for the finite dimensional samplers that form
motivation for the new family of algorithms that we propose here. 
From a technical point of view
the diffusion limit proved in this paper is similar to that proved
for the function space MALA in \cite{Pill:Stu:Thi:12}, extending to
the non-reversible case; however
significant technical issues arise which are not present in the
reversible case and, in particular, incorporating momentum flips
into the analysis, which occur for every rejected step, requires
new ideas.

\section{The \sol Algorithm} 
\label{sec:Hflow} 

In this section we introduce the \sol algorithm studied
in this paper.
We first describe the basic ideas from stochastic dynamics
underlying this work, doing so in the finite dimensional setting of $\cH=\bbR^N$, \textit{i.e.}, when the target measure  $\pi(q)$
 is a probability measure on  $\mathbb{R}^N$ of the form 
$$
\frac{d\pi}{d\pi_0}(q)\propto \exp(-\Psi(q)),
$$
where  $\pi_0$ is a mean zero Gaussian with covariance matrix  $\cal {C}$ and  $\Psi(q)$  is a function defined on $\mathbb{R}^N$. A key idea is to
work with an extended phase space in which the original variables
are viewed as \lq positions' and then \lq momenta' are added to
complement each position.
We then explain the advantages of working with \lq velocities' rather
than \lq momenta', in the large dimension limit. And then finally
we introduce our proposed algorithm, which is built on the
measure preserving properties of the second order Langevin
equation. As already mentioned, our algorithm will build on some basic facts about Hamiltonian mechanics. For a synopsys about the Hamiltonian formalism see Appendix C.

\subsection{Measure Preserving Dynamics in an Extended Phase Space}
 
Introduce the auxiliary variable $p$ (\lq momentum') and $\cM$ a
user-specified, symmetric positive definite `mass' matrix.
Let $\Pi_0'$ denote the Gaussian on $\bbR^{2N}$ defined
as the independent product of Gaussians $N(0,\cal{C})$ and $N(0,\cM)$ 
on the $q$ and $p$ coordinates respectively, and define $\Pi'$
by 
\begin{equation*}
\frac{d\Pi'}{d\Pi_0'}(q,p) \propto \exp\bigl(-\Psi(q)\bigr).
\end{equation*}
 A key point to notice is that the marginal of $\Pi'(q,p)$ with
respect to $q$ is the target measure $\pi(q).$
Define the Hamiltonian in $\Hf:\bbR^{2N} \to \bbR$ given by 
\begin{equation*}
 \Hf(q,p) = \tfrac{1}{2}\langle p, \cM^{-1} p\rangle
+ \tfrac{1}{2}\langle q, \mathcal{L} q \rangle +\Psi(q)\,
\end{equation*}
where $\mathcal{L}=\mathcal{C}^{-1}$.  The  corresponding canonical Hamiltonian
differential equation is given by
\begin{equation}
\label{eq:sthameq}
\frac{dq}{dt}=\frac{\partial \Hf}{\partial p}=\cM^{-1}p\ ,\quad
\frac{dp}{dt}=-\frac{\partial \Hf}{\partial q} =-\mathcal{L} q - D\Psi(q)\ .
\end{equation}
This equation preserves any smooth
function of $\Hf(q,p)$ and, as a consequence, the
Liouville equation corresponding to \eqref{eq:sthameq}
preserves the probability density of
$\Pi'(q,p)$, which is proportional to $\exp\big(-\Hf(q,p)\bigr).$
This fact is the basis for HMC methods \cite{Duane1987216}
which randomly sample momentum
from the Gaussian $N(0,\cM)$ and then run the Hamiltonian flow 
for $T$ time units; the resulting Markov chain on $q$ is $\pi(q)$
invariant. In practice the Hamiltonian flow must be integrated numerically,
but if a suitable integrator is used (volume-preserving and time-reversible)
then a simple accept-reject compensation corrects for numerical error.

Define
\begin{eqnarray*}
z=\left(
\begin{array}{c}
q\\
p
\end{array}
\right)
\end{eqnarray*}
and 
\be
J = \left( \begin{array}{cc}
0 & I  \\
-I & 0  \\
 \end{array} \right).
\ee
Then the Hamiltonian system can be written as
\begin{align}\label{eqn:1}
 {dz \over dt} 
 = J \,\nabla \Hf(z) 
 \end{align}
where, abusing notation, $\Hf(z):=\Hf(q,p).$
The equation \eqref{eqn:1} preserves the measure $\Pi'$. 

Now define the matrix
\be
\mathcal{K} = \left( \begin{array}{cc}
\mathcal{K}_1 & 0  \\
0 & \mathcal{K}_2  \\
 \end{array} \right)
\ee
where both $\mathcal{K}_1$ and $\mathcal{K}_2$ are symmetric. The following SDE also
preserves the measure $\Pi'$:
\be
 {dz \over dt} 
 =  - \mathcal{K}\, \nabla \Hf(z) + \sqrt{2\mathcal{K}} {dW \over dt} \;.
 \ee
Here $W = (W_1, W_2)$ denotes a standard Brownian motion on $\bbR^{2N}.$ 
This SDE decouples into two independent equations for $q$ and $p$;
the equation for $q$ is what statisticians term the Langevin equation \cite{Pill:Stu:Thi:12}, namely
$${dq \over dt}  =  - \mathcal{K}_1\bigl(\mathcal{L} q+D\Psi(q)\bigr) + \sqrt{2\mathcal{K}_1} {dW_1 \over dt} \;,$$
whilst the equation for $p$ is simply the Ornstein-Uhlenbeck process:
$${dp \over dt}  =  -\mathcal{ K}_2 \cM^{-1}p + \sqrt{2\mathcal{K}_2} {dW_2 \over dt} \;.$$
Discretizing the Langevin equation
(respectively the random walk found by ignoring the drift) 
and adding an accept-reject mechanism,
leads to the Metropolis-Adjusted Langevin (MALA)
(respectively the Random Walk Metropolis (RWM) algorithm).

A natural idea is to try and combine benefits of the HMC algorithm, 
which couples the position and momentum coordinates, with the MALA and 
RWM  methods. This thought experiment suggests considering
the second order Langevin equation\footnote{Physicists often refer 
to this as {\em the} Langevin equation for the choice 
$\cK_1 \equiv 0$ which leads to noise only appearing 
in the momentum equation.} 
\begin{align} \label{eqn:solang}
 {dz \over dt} &
 = J \,\nabla \Hf(z) - \mathcal{K}\, \nabla \Hf(z) + \sqrt{2\mathcal{K}} {dW \over dt} \;,
 \end{align}
which also preserves $\Pi'$ as a straightforward calculation with
the Fokker-Planck equation shows.
\subsection{Velocity Rather Than Momentum} \label{subs:velnomom}

Our paper is concerned with using the equation \eqref{eqn:solang}
to motivate proposals for MCMC. In particular we will be
interested in choices of the matrices $\cM$, $\mathcal{K}_1$ and $\mathcal{K}_2$ which lead
to well-behaved algorithms in the limit of large $N$. To this end we
write the equation \eqref{eqn:solang} in position and momentum 
coordinates as
\begin{align*}
{dq\over dt} &= \cM^{-1}p-\mathcal{K}_1\bigl(\mathcal{L} q+D\Psi(q)\bigr)+\sqrt{2\mathcal{K}_1}{dW_1\over dt},\\
{dp\over dt} &= -\bigl(\mathcal{L} q+D\Psi(q)\bigr)-\mathcal{K}_2 \cM^{-1}p+\sqrt{2\mathcal{K}_2}{dW_2\over dt}.
\end{align*}
In our subsequent analysis, which concerns the large $N$ limit, it turns out 
to be useful to work with velocity rather than momentum coordinates; this is
because the optimal algorithms in this limit are based on ensuring that the 
velocity and position coordinates all vary on the same scale.
 For this reason we introduce
$v=\cM^{-1}p$ and rewrite the equations as
\begin{align*}
{dq\over dt} &= v-\cK_1\bigl(\cL q+D\Psi(q)\bigr)+\sqrt{2\cK_1}{dW_1\over dt},\\
\cM{dv\over dt} &= -\bigl(\cL q+D\Psi(q)\bigr)-\cK_2v+\sqrt{2\cK_2}{dW_2\over dt}.
\end{align*}
In the infinite dimensional setting, \textit{i.e.}, when $\cal H$ is an infinite dimensional Hilbert  space, this equation is still well posed (see \eqref{eq:maininf} below and Theorem \ref{t:ode}). 
However in this case $W_1$ and $W_2$ are  cylindrical Wiener processes  on $\cal H$ (see Section \ref{subsec:notation}) and $\mathcal{L}=\mathcal{C}^{-1}$ is 
necessarily an unbounded operator  on ${\cal H}$
because the covariance operator $\cal {C}$ is trace class on $\cal { H}$. The unbounded operators introduce undesirable behaviour in the large $N$ limit 
when we approximate them; thus we choose
$\mathcal{M}$ and the $\mathcal{K}_i$ to remove the appearance of unbounded
operators. To this end 
we set $\cM=\mathcal{L}=\cC^{-1}$, $\mathcal{K}_1=\Gamma_1 \cC$ and $\mathcal{K}_2=\Gamma_2 \cC^{-1}$ and assume that $\Gamma_1$ and $\Gamma_2$ 
commute with $\cC$
to obtain the equations
\begin{subequations}
 \label{eq:maininf1}
\begin{align}
{dq\over dt} &= v-\Gamma_1\bigl(q+\cC D\Psi(q)\bigr)+\sqrt{2\Gamma_1 \cC}{dW_1\over dt},\\
{dv\over dt} &= -\bigl(q+\cC D\Psi(q)\bigr)-\Gamma_2v+\sqrt{2\Gamma_2 \cC}{dW_2\over dt},
\end{align}
\end{subequations}
or simply
\begin{subequations}
 \label{eq:maininf}
\begin{align}
{dq\over dt} &= v-\Gamma_1\bigl(q+\cC D\Psi(q)\bigr)+\sqrt{2\Gamma_1 }{dB_1\over dt},\\
{dv\over dt} &= -\bigl(q+\cC D\Psi(q)\bigr)-\Gamma_2v+\sqrt{2\Gamma_2 }{dB_2\over dt}.
\end{align}
\end{subequations}
In the above $B_1$ and $B_2$ are $\cal H$-valued Brownian motions with covariance operator $\cC$.
This equation is well-behaved in infinite dimensions provided that
the $\Gamma_i$ are bounded operators, and under natural assumptions
relating the reference measure, via its covariance $\cC$, and the log
density $\Psi$, which is  a real valued functional defined on an appropriate subspace of $\cal H$.  
Detailed  definitions and assumptions regarding \eqref{eq:maininf} are contained in the next Section \ref{sec:2}. Under such assumptions the function 
\begin{equation}
\label{eq:F}
F(q):=q+ \cC D\Psi(q)
\end{equation}
has desirable properties (see Lemma \ref{lem:lipschitz+taylor}), making the existence
theory for \eqref{eq:maininf} straightforward. 
We develop such theory in Section \ref{sec:infinite} -- see Theorem 
\ref{t:ode}.  Furthermore, in  Theorem \ref{t:flowprev} we will also prove that equation \eqref{eq:maininf} preserves the measure
$\Pi(dq,dv)$ defined by 
\begin{equation}
\label{eq:target2}
\frac{d\Pi}{d\Pi_0}(q,v) \propto \exp\bigl(-\Psi(q)\bigr),
\end{equation}
where $\Pi_0$ is the independent product of $N(0,\cC)$ with itself.
The measure $\Pi$ (resp. $\Pi_0$) is simply the measure
$\Pi'$ (resp. $\Pi_0'$) in the case $\cM=\cC^{-1}$ and rewritten
in $(q,v)$ coordinates instead of $(q,p)$.
In finite dimensions the invariance of $\Pi$  follows from
the discussions concerning the invariance of $\Pi'$.

\subsection{Function Space Algorithm}\label{sec:algorithm}

We note that the choice $\Gamma_1 \equiv 0$
gives the standard (physicists) Langevin equation
\begin{equation}
\label{eq:sop}
{d^2 q\over dt}+\Gamma_2{dq\over dt}+\bigl(q+\cC D\Psi(q)\bigr)=\sqrt{2\Gamma_2 \cC}{dW_2\over dt}.
\end{equation}
In this section we describe an MCMC method designed to
sample the measure $\Pi$ given by \eqref{eq:target2} and hence, by
marginalization, the measure $\pi$ given by \eqref{eqn:targmeas0}. 
The method is based on discretization of the second order Langevin
equation \eqref{eq:sop}, written as the hypo-elliptic
first order equation \eqref{eqn:g10} below.
In the finite dimensional setting a method closely related
to the one that we introduce was proposed in \cite{H91}; however
we will introduce different Hamiltonian solvers which are tuned
to the specific structure of our measure, in particular to the
fact that it is defined via density with respect to a Gaussian. 
We will be particularly interested in choices of parameters
in the algorithm which ensure that the output (suitability
interpolated to continuous time) behaves like \eqref{eq:sop}
whilst, as is natural for MCMC methods, exactly preserving the invariant
measure. This perspective on discretization of
the (physicists) Langevin equation in finite dimensions
was introduced in \cite{bou2011patch,Bou}.

 In position/velocity
coordinates, and using \eqref{eq:F}, \eqref{eq:maininf}  becomes 
\begin{align}\label{eqn:g10}
\begin{split}
{dq\over dt} &= v\,, \\
{dv\over dt} &= -F(q)-\Gamma_2v+\sqrt{2\Gamma_2 \cC}{dW_2\over dt}.
\end{split}
\end{align}
The algorithm we use is based on splitting \eqref{eqn:g10} into an Ornstein-Uhlenbeck (OU) process and a Hamiltonian ODE. The OU process is
\begin{align}\label{eqn:ousplit}
\begin{split}
{dq\over dt} &= 0\,, \\
{dv\over dt} &= -\Gamma_2v+\sqrt{2\Gamma_2 \cC}{dW_2\over dt},
\end{split}
\end{align}
and the Hamiltonian ODE is given by
\begin{align}\label{eqn:Hsplit}
\begin{split}
{dq\over dt} &= v\,, \\
{dv\over dt} &= -F(q) \;.
\end{split}
\end{align}
The solution of the OU process \eqref{eqn:ousplit} is denoted by $(q(t), v(t))=\Theta_0(q(0),v(0); \xi^t)$; here $\xi^t$ is a  mean zero Gaussian 
random variable with covariance operator 
$\cC\bigl(I-\exp(-2t\Gamma_2 )\bigr).$
Notice that the dynamics given by both \eqref{eqn:ousplit} 
and by \eqref{eqn:Hsplit}  preserve the 
target measure $\Pi$ given in \eqref{eq:target2}. This naturally suggests constructing an
 algorithm based on alternating the above two dynamics. However,
note that whilst \eqref{eqn:ousplit} can
 be solved exactly, \eqref{eqn:Hsplit} requires 
a further numerical approximation.
If the numerical approximation is based on a volume-preserving 
and time-reversible numerical integrator, then the accept-reject
criterion for the resulting MCMC algorithm can be easily 
expressed in terms of the energy differences in $\Hf$. 
A flow $\phi^t$  on $\RR^{2N}$ is said to be time-reversible if 
$\phi^t(q(0),v(0))=(q(t), v(t))$ implies 
$\phi^t(q(t), - v(t))=(q(0), -v(0))$.  Defintion of time-reversible and discussion of the
roles of time-reversible and volume-preserving integrators may be found in
\cite{sanz1994numerical}. 
 
To construct volume-preserving and time-reversible integrators
the Hamiltonian integration will be performed by a further 
splitting of \eqref{eqn:Hsplit}.
The usual splitting for the widely used Verlet method is via the velocity and the position
coordinates \cite{H91}. Motivated by our infinite dimensional setting, we 
replace the Verlet integration 
by the splitting method proposed in \cite{BPSS11};
this leads to an algorithm which is exact (no rejections)
in the purely Gaussian case where $\Psi \equiv 0.$
The splitting method proposed in  \cite{BPSS11}
is via the linear and nonlinear parts of the problem, 
leading us to consider the two equations
\begin{align}\label{rotation}
{dq\over dt} &= v\, ,\quad 
{dv\over dt} = -q,
\end{align}
with solution denoted as $(q(t), v(t)) = \sR^t (q(0), v(0))$; and 
\begin{align}\label{spl2}
 {dq\over dt} = 0 \;, \quad {dv\over dt} = - \cC D\Psi(q) ,
\end{align}
with solution denoted as
$(q(t), v(t)) = \Theta^t_1(q(0), v(0))$.
We note that the map
\be
\chi^t &= \Theta^{t/2}_1 \circ \sR^t \circ \Theta^{t/2}_1
\ee
is a volume-preserving and time-reversible second order
accurate approximation of the Hamiltonian ODE \eqref{eqn:Hsplit}.
We introduce the notation
\be 
\chi^t_{\tau} &= (\chi^t \circ \cdots \circ \chi^t), \quad \Big\lfloor \frac{\tau}{t}  
\Big\rfloor \; \,\mbox{times}
\ee to denote integration, using this method, up to time ${\tau}$.
This integrator can be made to preserve the measure $\Pi$ 
if appended with a suitable accept-reject rule as detailed below.
On the other hand the stochastic map 
$\Theta^t_0$ preserves $\Pi$ since it leaves $q$ invariant 
and since the OU process,
which is solved exactly, preserves $\Pi_0.$ 
We now take this idea to define our MCMC method. The infinite dimensional Hilbert space $\mathcal{H}^s \times \mathcal{H}^s$ in which the chain is constructed will be properly defined in the next section. Here we focus on the algorithm, which will be explained in more details and analyzed in Section 
\ref{sec:difflim}.

Define the operation $'$ so that
$v'$ is the velocity component of $\Theta^\delta_0(q,v).$
The preceding  considerations suggest that from point 
$(q^0,v^0) \in \cHHl$ we make the  proposal 
 \be
 (q_*^1,v_*^1) = \chi^h_{\tau} \circ \Theta^\delta_0 \, (q^0,v^0) 
 \ee 
and that the acceptance probability is given by
\be
\alpha(x^0,\xi\de): = 1 \wedge \exp \Bigl(\Hf\bigl(q^0,(v^0)'\bigr) - \Hf\bigl(q_*^1,v_*^1\bigr)\Bigr),
\ee 
where 
\be \label{haminf}
\Hf(q,v)=\frac{1}{2}\langle q, \C^{-1}q\rangle + \frac{1}{2}\langle v, \C^{-1}v\rangle + \Psi(q) \, ,
\ee
$\langle \cdot , \cdot \rangle$ denoting scalar product in $\h$.
One step of the resulting MCMC method is then defined
by setting
\begin{align}\label{eqn:alg}
\begin{split}
(q^1, v^1) &= (q_*^1,v_*^1)\,\,\,\,\, \qquad \mathrm{with \,probability}\, \,\alpha(x^0, \xi\de) \\
&= (q^0, -(v^0)') \,\quad \mathrm{otherwise}.
\end{split}
\end{align}
We will make further comments on this algorithm and on the expression \eqref{haminf} for the Hamiltonian   in Section \ref{sec:difflim}, see Remark 
\ref{rem:wellposaccprob}. Here it suffices simply to
note that whilst $H$ will be almost surely infinite, the energy
difference is well-defined for the algorithms we employ. 
We stress that when the proposal is rejected the chain does not remain in $(q^0, v^0)$ but it moves to 
$(q^0, -(v^0)')$; that is,  the position coordinate stays the same while the velocity coordinate is first evolved according to 
\eqref{eqn:ousplit} and then the sign is flipped. This flipping of the sign, needed to preserve reversibility, 
leads to some of the main technical differences with respect to \cite{PST13}; see Remark \ref{rem:flip}.
For the finite dimensional case with Verlet integration 
the form of the accept-reject mechanism and, in particular, the
sign-reversal in the velocity, was first derived in \cite{H91}
and is discussed in further detail in section 5.3 of \cite{neal2010mcmc}.
The algorithm \eqref{eqn:alg}  preserves $\Pi$ and we refer to it
as the \sol algorithm.  
Recalling that $v'$ denotes the velocity component of $\Theta_0^{\delta}(q,v)$,  we can equivalently use the notations 
$\alpha(x, \xi^{\delta})$ and $\alpha(q,{v}')$, for $x=(q,v)$ (indeed, by the definition of $\Theta_0\de$, $v'$ depends on $\xi\de$). With this in mind, the pseudo-code for the SOL-HMC is as follows. 

\rule{16cm}{.1pt}
\medskip

SOL-HMC in $\h^s$:


\bigskip

\begin{enumerate}
\item Pick  $(q^0, v^0) \in \h^s \times \h^s$ and set $k=0$;
\item given $(q^k, v^k)$, define $(v^k)'$ to be the $v$-component of  $\Theta^{\delta}_0(q^k,v^k)$ and calculate the proposal  $$(q^{k+1}_{*}, v^{k+1}_{*})= \chi_{\tau}^h  (q^k, (v^k)')\,;$$
\item define the acceptance probability $\alpha(q^k, (v^k)')$;
\item set $(q^{k+1}, v^{k+1})= (q^{k+1}_{*}, v^{k+1}_{*})$ with probability $\alpha(q^k, (v^k)')$; \\
otherwise set  $(q^{k+1}, v^{k+1})= (q^k, -(v^k)')$;
\item set $k \rightarrow k+1$ and go to (2). 
\end{enumerate}

\rule{16cm}{.1pt}

\begin{theorem} \label{t:alg}
Let Assumption \ref{ass:1} hold.
For any $\delta, h, {\tau} > 0$, the Markov chain defined by \eqref{eqn:alg} is
invariant with respect to $\Pi$ given by \eqref{eq:target2}.
\end{theorem}

\begin{proof}
See Appendix A. 
\end{proof}
\begin{remarks} 
\label{rem:need}

We first note that if
$\delta \rightarrow \infty$ then 
the algorithm \eqref{eqn:alg} is that introduced in the
paper \cite{BPSS11}. From this it follows that, 
if  $\delta=\infty$ and ${\tau} = h$, then the algorithm
is simply the funtion-space Langevin introduced in
\cite{Besk:etal:08}.

Secondly we mention that, in the numerical experiments reported later,
we will choose $\Gamma_2=I$. The solution of the OU process
\eqref{eqn:ousplit} for $v$ is thus given as
\begin{equation}
v(\delta)=(1-\iota^2)^{\frac12}v(0)+\iota w
\label{eq:preserve}
\end{equation} 
where $w \sim N(0,C)$ and $e^{-2\delta}=(1-\iota^2).$
The numerical experiments will be described in terms
of the parameter $\iota$ rather than $\delta.$
\end{remarks}

\section{Preliminaries} 
\label{sec:2}
In this section we detail the notation and the assumptions (Section \ref{subsec:notation} 
and Section \ref{sec:assumptions}, respectively) that we will use in the rest of the paper.  
\subsection{Notation}\label{subsec:notation}
Let $\Bigl(\h, \bra{\cdot, \cdot}, \|\cdot\|\Bigr)$ 
denote a separable Hilbert space of real valued functions with
the canonical norm derived from the inner-product. 
Let $\C$ be a  positive, trace class operator on $\h$ 
and $\{\phi_j,\lambda^2_j\}_{j \geq 1}$ be the eigenfunctions
and eigenvalues of $\C$ respectively, so that
\begin{equs}
\C\phi_j = \lambda^2_j \,\phi_j
\qquad \text{for} \qquad 
j \in \mathbb{N}.
\end{equs}
We assume a normalization under which $\{\phi_j\}_{j \geq 1}$ 
forms a complete orthonormal basis in $\h$.
For every $x \in \h$ we have the representation
$x = \sum_{j} \; x_j \phi_j$, where $x_j=\langle x,\phi_j\rangle.$
Using this notation, we define Sobolev-like spaces $\h^r, r \in \bbR$, with the inner products and norms defined by
\begin{equs}
\langle x,y \rangle_r = \sum_{j=1}^\infty j^{2r}x_jy_j
\qquad \text{and} \qquad
\|x\|^2_r = \sum_{j=1}^\infty j^{2r} \, x_j^{2}.
\end{equs}
Notice that $\h^0 = \h$. Furthermore
$\h^r \subset \h \subset \h^{-r}$ for any $r >0$.  
The Hilbert-Schmidt norm $\|\cdot\|_\C$ is defined as
\be
\|x\|^2_\C =\|\C^{-\frac12}x\|^2= \sum_{j=1}^{\infty} \lambda_j^{-2} x_j^2.
\ee
For $r \in \mathbb{R}$, 
let  $Q_r : \h \mapsto \h$ denote the operator which is
diagonal in the basis $\{\phi_j\}_{j \geq 1}$ with diagonal entries
$j^{2r}$, \textit{i.e.},
\be
Q_r \,\phi_j = j^{2r} \phi_j
\ee
so that $Q^{\frac12}_r \,\phi_j = j^r \phi_j$. 
The operator $Q_r$ 
lets us alternate between the Hilbert space $\h$ and the interpolation 
spaces $\h^r$ via the identities:
\be 
\bra{x,y}_r = \bra{ Q^{\frac12}_r x,Q^{\frac12}_r y } 
\qquad \text{and} \qquad
\|x\|^2_r =\|Q^{\frac12}_r x\|^2. 
\end{equs}
Since $\|Q_r^{-1/2} \phi_k\|_r = \|\phi_k\|=1$, 
we deduce that $\{Q^{-1/2}_r \phi_k \}_{k \geq 1}$ forms an 
orthonormal basis for $\h^r$. A function $y\sim N(0,\cC)$ can be  expressed as
\begin{align}\label{y1}
y=\sum_{j=1}^{\infty}
\lambda_j \rho_j \varphi_j  \qquad \mbox{with } \qquad \rho_j\stackrel{\mathcal{D}}{\sim}N(0,1) \,\,\mbox{i.i.d};
\end{align}
if $\sum_j \lambda_j^2 j^{2r}<\infty$ then $y$ can be equivalently written as
\begin{align}\label{eqn:y2}
y=\sum_{j=1}^{\infty}
(\lambda_j j^r) \rho_j (Q_r^{-1/2} \varphi_j)  \qquad \mbox{with } \qquad \rho_j\stackrel{\mathcal{D}}{\sim}N(0,1) \,\,\mbox{i.i.d}.
\end{align}
For a positive, self-adjoint operator $D : \h \mapsto \h$, its trace in $\h$ is defined as
\begin{equs}
\tr_{\h}(D) \;\eqdef\; \sum_{j=1}^\infty \bra{  \phi_j, D  \phi_j }.
\end{equs}
We stress that in the above $ \{ \varphi_j \}_{j \in \mathbb{N}} $ is an orthonormal basis for $(\h, \langle \cdot, \cdot \rangle)$. Therefore if 
$\tilde{D}:\h^r \rightarrow \h^r$, its trace in $\h^r$ is
\begin{equs}
\tr_{\h^r}(\tilde{D}) \;\eqdef\; \sum_{j=1}^\infty \bra{ Q_r^{-\frac{1}{2}} \phi_j, \tilde{D} Q_r^{-\frac{1}{2}} \phi_j }_r.
\end{equs}
Since $\tr_{\h^r}(\tilde{D})$ does not depend on the orthonormal basis,
the operator $\tilde{D}$ is said to be trace class in $\h^r$ if $\tr_{\h^r}(\tilde{D}) < \infty$ for
some, and hence any, orthonormal basis of $\h^r$. 

Because $\cC$ is defined on $\cH$, the covariance operator
\begin{equation}\label{Cs}
\cC_r=Q_r^{1/2} \cC Q_r^{1/2}
\end{equation}
is defined on $\h^r$. With this definition, 
for all the values of $r$ such that $\tr_{\h^r}(\cC_r)=\sum_j \lambda_j^2 j^{2r}< \infty$, we can think of $y$ as a mean zero Gaussian random variable with covariance operator $\cC$ in $\h$ and $\cC_r$ in $\h^r$ (see \eqref{y1} and \eqref{eqn:y2}).  In the same way, 
if $\tr_{\h^r}(\cC_r)< \infty $ then 
$$
B_2(t)= \sum_{j=1}^{\infty} \lambda_j \beta_j(t) \varphi_j= \sum_{j=1}^{\infty}\lambda_j j^r\beta_j(t) \hat{\varphi}_j,
$$
with $\{ \beta_j(t)\}_{j\in \mathcal{N}}$  a collection of  i.i.d. standard Brownian motions on $\mathbb{R}$, 
can be equivalently  understood as an $\h$-valued $\cC$-Brownian motion or as an  $\h^r$-valued $\cC_r$-Brownian motion.  
In the next section we will need the cylindrical Wiener process $W(t)$ which is
defined via the sum 
$$
W(t):= \sum_{j=1}^{\infty}\beta_j(t) \varphi_j.
$$ 
This process is ${\mathcal H}^r-$valued for any $r<-\frac12.$
Observe now that if $\{\hat{\varphi}_j\}_{j\in \mathbb{N}}$ is an orthonormal  basis of $\mathcal{H}^r$
 then, denoting $\mathcal{H}^r\times \mathcal{H}^r \ni \hat{\varphi}^1_j=( \hat{\varphi}_j,0)$ and 
$\mathcal{H}^r\times \mathcal{H}^r \ni \hat{ \varphi}^2_j=(0,\hat{ \varphi}_j)$, 
 $\mathfrak{F}=\{ \hat{ \varphi}^1_j, \hat{\varphi}^2_j \}_{j\in \mathbb{N}}$  is an orthonormal basis for 
$\mathcal{H}^r\times \mathcal{H}^r$.  Let  $\mathfrak{C}_r: \mathcal{H}^r\times \mathcal{H}^r \rightarrow \mathcal{H}^r\times \mathcal{H}^r$ be the diagonal operator such that 
\be
\mathfrak{C}_r\hat{\varphi}^1_j =(0,0),  \qquad 
\mathfrak{C}_r\hat{\varphi}^2_j=j^{2r}\lambda_j^2\hat{\varphi}^2_j=(0,\cC_r\hat{\varphi}_j)\qquad 
\forall j\in \mathbb{N}
\ee
and $\tilde{\mathfrak{C}}_r: \mathcal{H}^r\times \mathcal{H}^r \rightarrow \mathcal{H}^r\times \mathcal{H}^r$ be the diagonal operator such that 
\begin{align}\label{mfcs}
\tilde{\mathfrak{C}}_r\hat{\varphi}^1_j =j^{2r}\lambda_j^2 \hat{\varphi}_j^1=(\cC_r\hat{\varphi}_j,0), \qquad 
\tilde{\mathfrak{C}}_r\hat{\varphi}^2_j =j^{2r}\lambda_j^2 \hat{\varphi}_j^2=(0,\cC_r\hat{\varphi}_j) \qquad 
\forall j\in \mathbb{N}.
\end{align}
Consistently,  $B(t):=(0, B_2(t))$ will denote an  $\h^r \times \h^r$ valued Brownian motion with covariance operator $\mathfrak{C}_r$ and 
$\tilde{B}(t):=(B_1(t),B_2(t))$ will denote a $\h^r \times \h^r$ valued Brownian motion with covariance operator $\tilde{\mathfrak{C}}_r$. In other words, $B_1(t)$ and $B_2(t)$ are independent $\h^r$-valued $\cC_r$-Brownian motions.

Throughout we use the following notation.
\begin{itemize}
\item Two sequences of non-negative real numbers $\{\alpha_n\}_{n \geq 0}$ and $\{\beta_n\}_{n \geq 0}$ satisfy $\alpha_n \lesssim \beta_n$ 
if there exists a constant $K>0$ satisfying $\alpha_n \leq K \beta_n$ for all $n \geq 0$.
The notations $\alpha_n \asymp \beta_n$ means that $\alpha_n \lesssim \beta_n$ and $\beta_n \lesssim \alpha_n$.

\item 
Two sequences of non-negative real functions $\{f_n\}_{n \geq 0}$ and $\{g_n\}_{n \geq 0}$ defined on the same set $\Omega$
satisfy $f_n \lesssim g_n$ if there exists a constant $K>0$ satisfying $f_n(x) \leq K g_n(x)$ for all $n \geq 0$
and all $x \in \Omega$.
The notations $f_n \asymp g_n$ means that $f_n \lesssim g_n$ and $g_n \lesssim f_n$.

\item
The notation $\EE_x \big[ f(x,\xi) \big]$ denotes expectation 
with variable $x$ fixed, while the randomness present in $\xi$
is averaged out.
\end{itemize}

Also, let $\otimes_{\h^r}$
denote the outer product operator in $\h^r$
defined by
\begin{equs}
(x \otimes_{\h^r} y) z \eqdef \bra{ y, z}_r \,x 
\qquad \qquad 
\forall x,y,z \in \h^r.
\end{equs}
For an operator $A: \h^r \mapsto \h^l$, we denote its operator norm 
by $\opnm{\cdot}{r}{l}$ defined by
\begin{equs}
\opnm{A}{r}{l} \eqdef \sup_{\|x\|_r=1} \|A x\|_l.
\end{equs}
For self-adjoint $A$ and $r=l=0$ this is, of course, the spectral radius of $A$.
Finally, in the following we will consider the product space $\mathcal{H}^r \times \mathcal{H}^r$. The norm of 
$w=(w_1, w_2)\in \mathcal{H}^r \times \mathcal{H}^r$ is 
$$
\|w\|_{r\times r}^2 := \|w_1\|_{r}^2+ \|w_2\|_{r}^2. 
$$

\subsection{Assumptions} \label{sec:assumptions}
In this section we describe the assumptions on the covariance 
operator $\cC$ of the Gaussian measure $\pi_0 \dist \Normal(0,\cC)$ and the functional $\Psi$. We fix a distinguished exponent 
$s>0$ and assume that $\Psi: \mathcal{H}^s\rightarrow \RR$ and $\tr_{\mathcal{H}^s}(\mathcal{C}_s)<\infty$.
For each $x \in \h^s$ the derivative $\nabla \Psi(x)$
is an element of the dual $(\h^s)^*$ of $\h^s$ (dual with respect to the topology induced by the norm in $\h$), comprising 
the linear functionals on $\h^s$.
However, we may identify $(\h^s)^* = \h^{-s}$ and view $\nabla \Psi(x)$
as an element of $\h^{-s}$ for each $x \in \h^s$. With this identification,
the following identity holds:
\begin{equs}
\| \nabla \Psi(x)\|_{\mathcal{L}(\h^s,\RR)} = \| \nabla \Psi(x) \|_{-s}; 
\end{equs}
furthermore, the second derivative $\partial^2 \Psi(x)$ 
can be identified with an element 
of $\mathcal{L}(\h^s, \h^{-s})$.
To avoid technicalities we assume that $\Psi(x)$ is quadratically bounded, 
with first derivative linearly bounded and second derivative globally 
bounded. Weaker assumptions could be dealt with by use of stopping time 
arguments. 
\begin{assumptions} \label{ass:1}
The functional $\Psi$, covariance operator $\cC$ and the operators $\Gamma_1, \Gamma_2$ satisfy the following assumptions.
\begin{enumerate}
\item {\bf Decay of Eigenvalues $\lambda_j^2$ of $\cC$:}
there exists a constant $\kappa > \frac{1}{2}$ such that
\begin{equs}
\lambda_j \asymp j^{-\kappa}.
\end{equs}

\item {\bf Domain of $\Psi$:}
there exists an exponent $s \in [0, \kappa - 1/2)$ such 
that $\Psi$ is defined everywhere on $\h^s$.

\item {\bf Size of $\Psi$:} 
the functional $\Psi:\h^s \to \RR$ satisfies the growth conditions
\begin{equs}
0 \quad\leq\quad \Psi(x) \quad\lesssim\quad 1 +  \|x\|_s^2 . 
\end{equs}

\item {\bf Derivatives of $\Psi$:} 
The derivatives of $\Psi$ satisfy
\begin{align}\label{der-s}
\| \nabla \Psi(x)\|_{-s} \quad\lesssim\quad 1 + \|x\|_s  
\qquad \text{and} \qquad
\opnm{\partial^2 \Psi(x)}{s}{-s} \quad\lesssim\quad 1.
\end{align}

\item {\bf Properties of the $\Gamma_i$:} The operators
$\Gamma_1, \Gamma_2$ commute with $\cC$ and are bounded 
linear operators from $\h^s$ into itself.
\end{enumerate}
\end{assumptions}

\begin{remark}
The condition $\kappa > \frac{1}{2}$ ensures that $\tr_{\h^r}(\cC_r)  < \infty$
for any $r < \kappa - \frac{1}{2}$: this implies that
$\pi_0(\h^r)=1$ 
for any $\tau > 0$ and  $r < \kappa - \frac{1}{2}$.
\label{rem:one}
\end{remark}

\begin{remark}
The functional $\Psi(x)  = \frac{1}{2}\|x\|_s^2$ is defined on $\h^s$ and its derivative at $x \in \h^s$
is given by $\nabla \Psi(x) = \sum_{j \geq 0} j^{2s} x_j \phi_j \in \h^{-s}$ with  
$\|\nabla \Psi(x)\|_{-s} = \|x\|_s$. The second derivative $\partial^2 \Psi(x) \in \mathcal{L}(\h^s, \h^{-s})$ 
is the linear operator that maps $u \in \h^s$ to $\sum_{j \geq 0} j^{2s} \bra{u,\phi_j} \phi_j \in \h^{-s}$:
its norm satisfies $\| \partial^2 \Psi(x) \|_{\mathcal{L}(\h^s, \h^{-s})} = 1$ for any $x \in \h^s$.
\end{remark}

\noindent
The Assumptions \ref{ass:1} ensure that the functional $\Psi$ behaves 
well in a sense made precise in the following lemma.

\begin{lemma} \label{lem:lipschitz+taylor}
Let Assumptions \ref{ass:1} hold.
\begin{enumerate}
\item 
The function $F(x)$ given by \eqref{eq:F} is globally 
Lipschitz on $\h^s$:
\begin{equs} 
\|F(x) - F(y)\|_s \;\lesssim \; \|x-y\|_s
\qquad \qquad \forall x,y \in \h^s.
\end{equs}

\item
The second order remainder term in the Taylor expansion of $\Psi$ satisfies
\begin{align} \label{e.taylor.order2}
\big| \Psi(y)-\Psi(x) - \bra{\nabla \Psi(x), y-x} \big| \lesssim \|y-x\|_s^2
\qquad \qquad \forall x,y \in \h^s.
\end{align}
\end{enumerate}
\end{lemma}
\begin{proof}
See \cite{Matt:Pill:Stu:11,PST13}.
\end{proof}

\section{SPDE Theory} 
\label{sec:infinite}

In this section we study the SDE \eqref{eq:maininf} 
in the infinite dimensional Hilbert space setting; we
work under the assumptions specified in the previous section. 
Recall that our goal is to sample the measure $\pi$ in
\eqref{eqn:targmeas0},
but that we have extended our state space to obtain the measure 
$\Pi$ given by \eqref{eq:target2}, with $q$ marginal given
by $\pi$. Here $\Pi_0$ is the independent
product of $\pi_0=N(0,\cC)$ with itself in the $q$ and $p$
coordinates. The finite dimensional arguments in Section
\ref{sec:Hflow} show that the
equation \eqref{eq:maininf} preserves $\Pi_0$. 
The aim of this section is to show that these steps all make
sense  in the infinite dimensional context, under
the assumptions laid out in the previous section.

\begin{theorem}
\label{t:ode}
Let Assumption \ref{ass:1} hold. Then, for any initial condition
$\bigl(q(0),v(0)\bigr)\in \cHHl$, any $T>0$ and almost every
$\h^s \times \h^s$-valued $\tilde{\mathfrak{C}}_s$-Brownian motion $\tilde{B}(t)=(B_1(t), B_2(t))$, there exists a
unique solution of the SDE (\ref{eq:maininf}) in the space
$C([0,T],\cHHl)$. Furthermore, the It\^o map
$(B_1,B_2) \in C\bigl([0,T];\cHHl\bigr) \mapsto (q,v)
\in C\bigl([0,T];\cHHl\bigr)$ is Lipschitz. 
\end{theorem}

\begin{proof} If we define
\begin{eqnarray*}
x=\left(
\begin{array}{c}
q\\
v
\end{array}
\right),
\end{eqnarray*}
together with the operator 
\be
\Gamma = \left( \begin{array}{cc}
\Gamma_1 & 0  \\
0 & \Gamma_2  \\
 \end{array} \right),
\ee
then equation \eqref{eq:maininf}
takes the form
\begin{align}\label{eqn:simple}
 {dx \over dt} 
 =  G(x) + \sqrt{2\Gamma} {d\tilde{B} \over dt} \;
 \end{align}
where
\begin{eqnarray}\label{GGGGG}
G(x)=\left(
\begin{array}{c}
v-\Gamma_1 F(q)\\
-F(q)-\Gamma_2 v
\end{array}
\right).
\end{eqnarray}
A solution of \eqref{eq:maininf} satisfies the
integral equation
\be
x(t)=x_0+\int_0^t G\bigl(x(s)\bigr)ds+\sqrt{2\Gamma}\tilde{B}(t),
\ee
where $x(0)=x_0.$
By virtue of Lemma \ref{lem:lipschitz+taylor}
we see that $G: \cHHl \to \cHHl$ is globally
Lipschitz. Furthemore, Remark \ref{rem:one} shows
that $\tilde{B} \in C\bigl([0,T];\cHHl\bigr)$ almost surely.
To prove existence and uniqueness of a solution we consider
the map $\Xi: C\bigl([0,T];\cHHl\bigr) \mapsto
C\bigl([0,T];\cHHl\bigr)$ defined by
\be 
\Xi(x)(t):=x_0+\int_0^t G\bigl(x(s)\bigr)ds+\sqrt{2\Gamma}\tilde{B}(t).
\ee
Since $F$ is globally Lipschitz from $\h^s$ into itself,
it follows that $G$ is globally Lipschitz 
from $\h^s \times \h^s$ into itself.
This in turn implies that $\Xi$
is Lipschitz and that, furthermore, the Lipschitz constant
may be made less than one, by choosing $t$ sufficiently small.
From this existence and uniqueness of a solution follows
by the contraction mapping principle, on time-intervals sufficiently
small. The argument may then be repeated on successive time-intervals
to prove the result on any time-interval $[0,T].$

Now let 
\begin{align}\label{upsilon}
\Upsilon: (x_0,\tilde{B}) \in \cHHl \times C\bigl([0,T];\cHHl\bigr)
\mapsto  x \in C\bigl([0,T];\cHHl\bigr).
\end{align}
The arguments used in Lemma 3.7 of \cite{Matt:Pill:Stu:11}
show that $\Upsilon$ is Lipschitz continuous and
hence the desired properties of the It\^o map follow. 
\end{proof}

For $N \in \mathbb{N}$, let $\h^N$ denote the linear span of 
the first $N$ eigenfunctions of $\cC$, $P^N : \h \mapsto \h^N$ denote
the projection map and $\Psi^N=\Psi \circ P^N$. Define
$Q^N = I - P^N$. Recall Equations \eqref{eq:maininf1}. Let 
$\Upsilon^N$ denote the Ito map obtained by replacing $D\Psi$ by $P^N D\Psi^N$ in \eqref{eq:maininf1}. 

The following is the key result of this section. 
Our choices of measure (\ref{eq:target2}) and dynamics
(\ref{eq:maininf1}) have been coordinated to
ensure that the resulting stochastic dynamics preserves $\Pi$:
\begin{theorem}
\label{t:flowprev} 
For any initial condition $\bigl(q(0),v(0)\bigr)\sim \Pi$
and any $T>0,$ the equation \eqref{eq:maininf1} preserves  
$\Pi:$ $\bigl(q(T),v(T)\bigr)\sim \Pi.$ 
\end{theorem}
\begin{proof} The proof follows along the lines of Theorem 3.1 of \cite{BPSS11}.
The key idea is to exploit the fact that for finite dimensional  $\h$, the
 invariance of $\Pi$ under the dynamics $(q(T),v(T))$ follows easily. From this,
 the invariance for an infinite dimensional $\h$ follows from an application
 of the dominated convergence theorem which we outline below.
 
 We let $\bbW$ denote the
Weiner measure on  $X=C([0,T];\cHHl)$ induced by Brownian motions
with covariance the same as that of  $\Pi_0.$
For any continuous, bounded function $g: \h^s \times \h^s \mapsto \mathbb{R}$ and $T > 0$, 
we need to show that 
\begin{align}
\int_{\h \times X} g&\Bigl(\Upsilon \bigl(q,v,W\bigr)\Bigr) \exp{\big(-\Psi(q)\big)}\, d\Pi_0(q,v) d\bbW(W) \nonumber\\
&=
\int_{\h} g(q,v) \exp{\big(-\Psi(q)\big)}\, d\Pi_0(q,v). 
\label{eq:ONE}
\end{align}
First, we claim that for any $N \in \mathbb{N}$,
\begin{align}
\int_{\h \times X} g&\Bigl(\Upsilon^N\bigl(q,v,W\bigr)\Bigr) \exp{\big(-\Psi^N(q)\big)}\, d\Pi_0(q,v) d\bbW(W) \nonumber \\
&= \int_{\h} g(q,v) \exp{\big(-\Psi^N(q)\big)}\, d\Pi_0(q,v). 
\label{eq:TWO}
\end{align}
This follows from the fact that the flow $\Upsilon^N$  preserves the 
invariant measure proportional to $\exp(-\Psi^N) \Pi_0$ as obtained below in
Lemma \ref{lem:upnpresinv}. 


In Lemma \ref{lem:unconv} below, we will show that $\Upsilon^N$  converges pointwise
to $\Upsilon$. Thus by the continuity of $g$, $g\Bigl(\Upsilon^N\bigl(q,v,W\bigr)\Bigr)$ converges
pointwise to $g\Bigl(\Upsilon\bigl(q,v,W\bigr)\Bigr)$. Clearly, $\exp(-\Psi^N(q))$ converges to $\exp(-\Psi(q))$
pointwise.  Since $g$ is bounded and $\Psi, \Psi^N$ are positive, by the dominated convergence theorem the right (resp. left) hand side of of \eqref{eq:TWO} converges to the right (resp. left) hand side of \eqref{eq:ONE} and the claim follows.
\end{proof}

\begin{lemma} \label{lem:upnpresinv}
Let Assumption \ref{ass:1} hold.
The measure $\Pi^N \propto \exp(-\Psi^N) \Pi_0$ factors as the product
of two measures on $P^N \h^s$ and $Q^N \h^s$. The measure $\Pi^N\propto \exp(-\Psi^N) \Pi_0$ is preserved by $\Upsilon^N$.
\end{lemma}

\begin{proof}
By construction the measure $\Pi_0$ factors as the product of  two measures
$\mu_0 = {N}(0,P^N\cC P^N)$ and $\mu^{\perp}_0 = {N}(0, Q^N \cC Q^N)$. Since $\Psi^N$ is $0$ on $Q^N$, it follows that $\Pi^N$ factors into $\mu_1 \propto \exp(-\Psi^N) \mu_0$ on 
$P^N \h^s$ and $\mu^{\perp}_1 = \mu^{\perp}_0$ on $Q^N \h^s$.

Now, as explained in Section \ref{sec:Hflow} for any $N$, $\mu_1$ is invariant
for  $P^N \Upsilon^N$. Also setting $\Psi = 0$ in \eqref{eq:maininf1} results in an OU flow on $\h^s$ for which $\Pi_0$ is invariant. Thus if $ D\Psi$ is replaced by $P^N D\Psi^N$ in \eqref{eq:maininf1}, the resulting
flow on $Q^N$ is an Orstein-Uhlenbeck process with invariant measure $\mu^{\perp}_1$. 
Since $\Pi^N$ is a product of $\mu_1$ and $\mu^{\perp}_1$, the result follows.
\end{proof}
 
The following result shows the pointwise convergence of $\Upsilon^N$ to $\Upsilon$.
\begin{lemma} \label{lem:unconv}
Let Assumption \ref{ass:1} hold.
As $N \rightarrow \infty$, $\Upsilon^N(x_0,\tilde{B})$ converges to $\Upsilon(x_0,\tilde{B})$ 
for every $(x_0, \tilde{B}) \in  \cHHl \times C\bigl([0,T];\cHHl\bigr)$.
\end{lemma}

\begin{proof}
Proceeding similarly as in Theorem \ref{t:ode},  set
\be 
 {dx^N \over dt} 
 =  G^N(x) + \sqrt{2\Gamma} {d\tilde{B} \over dt} \;
 \ee
where
\begin{eqnarray} \label{eqn:Gx1}
G^N(x)=\left(
\begin{array}{c}
v-\Gamma_1 F^N(q)\\
-F^N(q)-\Gamma_2 v
\end{array}
\right)
\end{eqnarray}
with $F^{N}(q) = q+\cC P^N D\Psi^N(q)$.
Let $x(t)$ denote the solution 
of \eqref{eq:maininf1} and $x^N$ above satisfies
\begin{align} \label{eqn:inteqn1}
x^N(t)=x_0+\int_0^t G^N\bigl(x^N(s)\bigr)ds+\sqrt{2\Gamma}\tilde{B}(t),
\end{align}
where $x(0)=x_0.$ Set $e = x - x^N$. The pointwise convergence
of $\Upsilon^N$ to $\Upsilon$ is established by showing that
$ e \rightarrow 0$ in the path space $C\bigl([0,T];\cHHl\bigr)$.
We first decompose:
\begin{align}\label{eqn:Gxdec}
G(x) - G^N(x^N) = \big(G(x) - G^N(x)\big) + \big(G^N(x) - G^N(x^N)\big) \;.
\end{align}
Next, it can be shown that  $G^N$ is globally Lipschitz with the Lipschitz constant $L$
independent of $N$ (see \cite{Pill:Stu:Thi:12}, Lemma 4.1). Thus we have $
\|G^N(x(t)) - G^N(x^N(t))\|_s \leq L \|e(t)\|_s \;.$
Combining this bound with \eqref{eqn:inteqn1} and \eqref{eqn:Gxdec},
\be
\|e(t)\|_s \leq \int_0^t L \|e(u)\|_s\, du + \int_0^t \|G(x(u)) - G^N(x(u))\|_s\, du \;.
\ee
Thus by Gronwall's inequality, it suffices to show that
\be
 \sup_{0 \leq t \leq T} \|G(x(t)) - G^N(x(t))\|_s \rightarrow 0
\ee
 as $N \rightarrow \infty$. To this end, write
\be
F(x) - F^N(x) =  (\cC D\Psi(x) - \cC P^N D \Psi(x)) + 
(\cC P^N D \Psi(x) - \cC P^N D \Psi^N(x)) \;.
\ee
Since $C D\Psi$ is globally Lipschitz,
\begin{align}\label{eqn:Gxnbound}
\|G(x(t)) - G^N(x(t))\|_s \lesssim \|(I - P^N)CD\Psi(x(t))\|_s  + 
\| (I - P^N)x(t)\|_s \;.
\end{align}
From the existence of a global solution for \eqref{eq:maininf1} as shown in Theorem \ref{t:ode},
it follows that $\sup_{0 \leq t \leq T} \|x(t)\|_s < \infty$.
Thus from \eqref{eqn:Gxnbound} we infer that
$\sup_{0 \leq t \leq T} \|G(x(t)) - G^N(x(t))\|_s \rightarrow 0$, and the claim follows.
\end{proof}

\section{Diffusion Limit of Algorithms}\label{sec:difflim}
The main result of this section is the diffusion limit Theorem \ref{thm:difflim}: using the prescription \eqref{eqn:alg} and setting 
$\delta=h={\tau}$, we construct a sequence of Markov chains $x\kd$ (i.e., for every fixed delta, $\{x\kd\}_k$ is a Markov chain) and consider   the process $z\de(t)$  which is the continuous time interpolant of  the chain $x\kd$. Then $z\de(t)$ converges to the solution of the SDE \eqref{eqn:thmdifflim}, which is a specific instance of \eqref{eqn:simple}, when $\Gamma_1=0$.  By Theorem \ref{t:flowprev}, the flow \eqref{eqn:thmdifflim}  preserves the measure $\Pi$ defined in \eqref{eq:target2}. 

More precisely, 
 for $q,v\in \Hel$, let  $x\in \cHHl $ denote the pair $x=(q,v)$; we recall that  the norm of $x$ is then
$$
\|x\|^2_{s\times s} :=\|q\|_s^2+\|v\|_s^2.
$$
With the algorithm described in Section \ref{sec:algorithm}, taking $\delta=h={\tau}$  we construct the Markov chain 
$x\kpod:=(q^{k+1,\delta}, v^{k+1,\delta})$ as follows
\begin{align}\label{eqn:algK}
\begin{split}
(q^{k+1,\delta}, v^{k+1,\delta}) &= (q_*^{k+1,\delta},v_*^{k+1,\delta}) \quad \mathrm{with \,probability}\, \,\alpha^{k,\delta} \\
&= (q{\kd}, -(v{\kd})') \,\quad \mathrm{otherwise},
\end{split}
\end{align}
where 
$$
\alpha\kd=\alpha(x\kd,\xi\de): = 1 \wedge \exp \Bigl(\Hf\bigl(q\kd,(v\kd)'\bigr) - \Hf\bigl(q_*\kpod,v_*\kpod \bigr)\Bigr).
$$
We specify that in the above
$$
(q_*^{k+1,\delta},v_*^{k+1,\delta})  = \chi^{\delta} \circ \Theta^\delta_0 \, (q{\kd},v{\kd}) \; \quad \mbox{and} \quad 
((q{\kd})', -(v{\kd})')=\left(q{\kd}, -\mathcal{P}_v\left(\Theta_0^{\delta}(q\kd,v{\kd})\right)\right),
$$
where for $x\in \cHHl$, we denote by  $\mathcal{P}_q(x)$ and  $\mathcal{P}_v(x)$ the projection  of $x$ on the $q$ and $v$ component, respectively. Notice that introducing  $\gamma\kd \sim \textup{Bernoulli}(\alpha\kd)$, the algorithm \eqref{eqn:algK} can be 
also written as
\begin{align*}
(q^{k+1,\delta}, v^{k+1,\delta}) &=\gamma\kd (q_*^{k+1,\delta},v_*^{k+1,\delta})+ (q{\kd}, -(v{\kd})').
\end{align*}

Following \cite{PST13}, we consider the piecewise linear and the piecewise constant  interpolant of the chain $x\kd$, $z^{\delta}(t)$
and $\bar{z}^{\delta}(t)$, respectively:
\begin{align}\label{eqn:continter}
z^{\delta}(t):=\frac{1}{\delta}(t-t_k)x\kpod+\frac{1}{\delta}(t_{k+1}-t)x\kd, 
\qquad t_k\leq t <t_{k+1}, \mbox{  } t_k=k\delta ,
\end{align}
\begin{align}\label{eqn:constinter}
\bar{z}^{\delta}(t):=x\kd \qquad t_k\leq t <t_{k+1}, \mbox{  } t_k=k\delta.
\end{align}
Decompose the chain $x\kd$ into its drift and martingale part:
$$
x\kpod=x\kd+\delta G\de(x\kd) +\sqrt{2\delta \mathcal{S}}M\kd
$$
where 
$$
\mathcal{S}=
\left[
\begin{array}{cc}
\mathrm{Id} & 0\\
0 &\Gamma_2
\end{array}
\right],
$$
\begin{align}
& G\de(x)  :=\frac{1}{\delta} \EE_x\left[x\kpod-x\kd\vert x\kd=x\right],\label{eqn:ddelta}\\
& M\kd  := \frac{\mathcal{S}^{-1/2}}{\sqrt{2\delta}}\left( x\kpod-x\kd - \delta G\de(x\kd)\right)
\label{eqn:Mdelta}\\
& M\de(x)  := \EE\left[ M\kd \vert x\kd=x\right].\label{mdeltax}
\end{align}
Notice that with this definition, if $\mathcal{F}\kd$ is the filtration generated by $\{x^{j,\delta}, \gamma^{j,\delta}, \xi\de, j=0, \dots, k\}$, we have 
$\EE[M\kd\vert \mathcal{F}\kd ]=0$. 
Also, let us introduce the rescaled noise process
\begin{align}\label{rescnoise}
\tilde{B}\de(t):=\sqrt{2\mathcal{S}\delta}\sum_{j=0}^{k-1} M^{j,\delta}+\sqrt{\frac{2\mathcal{S}}{\delta}}(t-t_k)
M\kd, \qquad t_k \leq t < t_{k+1}.
\end{align}
A simple calculation, which we present in Appendix A, shows  that
\begin{align}\label{mapB}
z\de(t)=\Upsilon(x_0,\hat{B}\de),
\end{align}
where $\Upsilon$ is the map defined in \eqref{upsilon} and $\hat{B}\de$ is the rescaled noise process $\tilde{B}^{\delta}$ plus a term which we will show to be small:
\be 
\hat{B}\de(t):=\tilde{B}\de(t)+\int_0^t \left[ G\de(\bar{z}\de(u))- G(z\de(u))\right] du;
\ee
we stress that in the above and throughout this section the map $G(x)$ is as in \eqref{GGGGG} with $\Gamma_1=0$.

Let $B_2(t)$ be an $\mathcal{H}^s$-valued  $\cC_s$-Brownian motion (we recall that the covariance operator $\mathcal{C}_s$ has been defined in \eqref{Cs}) and $\mathcal{H}^s\times \mathcal{H}^s \ni B(t)=(0,B_2(t))$. Recall the SPDE
\eqref{eq:maininf} written in the form \eqref{eqn:simple}.
The main result of this section is the following diffusion limit of the Markov
chain \eqref{eqn:algK} to \eqref{eqn:simple}.

\begin{theorem}[Diffusion limit] \label{thm:difflim}
Let Assumption \ref{ass:1} hold and
let $(\cH^s, \langle \cdot, \cdot \rangle_s) $ be a separable Hilbert space,  $x\kd$ be the Markov chain \eqref{eqn:algK} starting at $x^{0,\delta}=x_0 \in \mathcal{H}^s \times \mathcal{H}^s$ and 
 let $z\de(t)$ be the process  defined by  \eqref{eqn:continter}.  If Assumption \ref{ass:1} holds then $z\de(t)$ converges weakly in $C([0,T];  \cH^s\times \cH^s)$  to the solution $z(t)\in  \cH^s\times \cH^s$ of the stochastic differential equation
\begin{align}\label{eqn:thmdifflim}
\begin{split}
 dz(t)&=G(z)dt + \sqrt{2\Gamma}\, dB(t)\\
z(0)&= x_0.
\end{split}
\end{align}
\end{theorem}
The diffusion limit can be proven as a consequence of  \cite[Lemma 3.5]{PST13}.  Proposition \ref{lemma:difflimit} below is a slightly more general version 
of \cite[Lemma 3.5]{PST13}.
\begin{proof}[Proof of Theorem \ref{thm:difflim}] Theorem \ref{thm:difflim} follows as a consequence of Proposition
\ref{lemma:difflimit} and Lemma
\ref{lemma:ass-cond} below.
\end{proof}
 Consider the following conditions:

\begin{cond}\label{drift-martCond}
The Markov chain $x\kd \in  \cH^s\times \cH^s$ defined in \eqref{eqn:algK} satisfies
\begin{itemize}
\item {\bf Convergence of the approximate drift.} There exist a globally Lipshitz function \linebreak
$G: \cH^s\times \cH^s \rightarrow   \cH^s\times \cH^s$,  a real number $a>0$ and an integer $p\geq 1$ such that
\begin{align}\label{dde-G}
\|G\de(x)-G(x)\|_{s\times s}\les \delta^a (1+ \|x\|_{s\times s}^p).
\end{align}

\item {\bf Size of the increments.} There exist a real number $r>0$ and an integer $n\geq 1$ such that
\begin{align}\label{sizeincrem}
\EE\left[ \nors{x\kpod -x\kd}\vert x\kd=x  \right]\lesssim \delta^{r}(1+ \nors{x}^{n}).
\end{align}

\item {\bf A priori bound.} There exists a real number $\epsilon$ such that $1-\epsilon +(a\wedge r)>0$ (with $a$ and $r$ as in 
\eqref{dde-G} and \eqref{sizeincrem}, respectively) and the following bound holds:
\begin{align}\label{dde}
\sup_{\delta\in (0, 1/2)}\left\{  \delta^{\epsilon} \EE \left[ \sum_{k\delta\leq T} \| x\kd \|_{s\times s}^{p \vee n} \right]    
\right\}< \infty.
\end{align}

\item {\bf Invariance principle.} As $\delta$ tends to zero the sequence of processes $\tilde{B}\de$ defined in \eqref{rescnoise} converges weakly in 
$C([0,T];\mathcal{H}^s \times \mathcal{H}^s)$  to the Brownian motion 
$\mathcal{H}^s\times \mathcal{H}^s \ni B=(0,B_2)$ where $B_2$ is a $\mathcal{H}^s$-valued, $\cC_s$-Brownian motion.
\end{itemize}
\end{cond}

\begin{remark}\label{Remark} \textup{
Notice that if \eqref{dde-G} holds for some $a>0$ and $p\geq 1$, then  
\begin{align}\label{sizeincremf}
\nors{\EE [ x\kpod -x\kd\vert x\kd ]}\lesssim \delta (1+ \nors{x}^{p})
\end{align}
and
\begin{align}\label{nonldrift}
\nors{G\de (x)}\les 1+\nors{x}^p. 
\end{align}
Indeed 
$$
\nors{\EE [ x\kpod -x\kd\vert x\kd =x]} = \delta \nors{G\de(x)} \leq
\delta\nors{G\de (x)-G(x)}+\delta \nors{G(x)} \les \delta (1+ \nors{x}^p), 
$$
having used the Lipshitzianity of the map $G(x)$. Analogously one can obtain \eqref{nonldrift} as well.
}
\end{remark}
\begin{prop}\label{lemma:difflimit}
Let Assumption \ref{ass:1} hold and let
$(\cH^s, \langle \cdot, \cdot \rangle_s) $ be a separable Hilbert space and $x\kd$ a sequence of $ \cH^s\times \cH^s$ valued
 Markov chains with $x^{0,\delta}=x_0.$
Suppose the drift martingale decomposition  \eqref{eqn:ddelta}- \eqref{eqn:Mdelta}of $x\kd$  satisfies Condition \ref{drift-martCond}. Then the sequence of interpolants $z\de(t)$ defined in \eqref{eqn:continter}  converges weakly in $C([0,T];  \cH^s\times \cH^s)$  to the solution $z(t)\in  \cH^s\times \cH^s$ of the stochastic differential equation \eqref{eqn:thmdifflim}.
\end{prop}
\begin{proof}
Thanks to the Lipshitzianity of the map  $\Upsilon$ in \eqref{mapB} (see Theorem \ref{t:ode}), the proof is analogous to the proof of 
\cite[Lemma 3.5]{PST13}. We sketch it in Appendix A. 
\end{proof}

\begin{lemma}\label{lemma:ass-cond}
Let Assumption \ref{ass:1} hold and let
$x\kd$ be the Markov chain \eqref{eqn:algK} starting at $x^{0,\delta}=x_0 \in \mathcal{H}^s \times \mathcal{H}^s$. Under Assumption \ref{ass:1} the drift martingale decomposition of
 $x\kd$,  \eqref{eqn:ddelta}- \eqref{eqn:Mdelta}, satisfies Condition \ref{drift-martCond}.
\end{lemma}

The remainder of this section is devoted to proving Lemma \ref{lemma:ass-cond}, which is needed to prove
Theorem \ref{thm:difflim}. First, in Section \ref{sec:Basic Estimates} we  list and explain several preliminary technical lemmata, which will be proved in Appendix B.  The main one is Lemma \ref{lemma:alpha-1}, where we study the acceptance probability.   Then, in Section  \ref{sec:AnalysisDrift} and Section \ref{sec:analysisNoise}, we prove Lemma \ref{lemma:ass-cond}; in order to prove such a lemma we need to show that if  Assumption \ref{ass:1} holds, the four conditions listed in Condition \ref{drift-martCond} are satisfied by the chain $x\kd$. To this end, 
Lemma \ref{lemma:d-G} proves that \eqref{dde-G} holds with $a=1$ and $p=6$; Lemma 
\ref{Lemmasize} shows that \eqref{sizeincrem} is satisfied with $r=1/2$ and $n=6$; the a priori bound \eqref{dde} 
is proved to hold for $\epsilon=1$ and for any power of $\nors{x\kd}$ in Lemma \ref{lemma:sup}; finally,  Lemma \ref{lemmanoise} is the invariance principle.

\begin{proof}[Proof of Lemma \ref{lemma:ass-cond}]  Lemma \ref{lemma:ass-cond} follows as a consequence of Lemma \ref{lemma:d-G}, Lemma \ref{Lemmasize}, 
 Lemma \ref{lemma:sup} and Lemma \ref{lemmanoise}.
\end{proof}

\subsection{Preliminary Estimates}\label{sec:Basic Estimates}
We first analyse the acceptance probability. 
 Given the current state of the chain  $x\kd=x=(q,v)$, the acceptance probability of the proposal $(q_*, v_*)$ is  
\begin{align}\label{alphade}
\alpha\de:=\alpha^{0,\delta}(x,\xi\de)= 1\wedge \exp\left(\Hf (q, v')-\Hf (q_*, v_*)\right)=
1\wedge \exp\left(\Delta \Hf (q, v')\right). 
\end{align}
Similarly, we denote 
$$
\gamma\de:=\gamma^{0,\delta}\sim\textup{Bernoulli}(\alpha\de).
$$
For an infinite dimensional Hilbert space setting, the matter of the well-posedness of the expression for the acceptance probability is not obvious; we comment on this below. 
\begin{remark}\label{rem:wellposaccprob}\textup{
Before proceeding to the analysis, let us make a few observations about the expression \eqref{alphade} for the acceptance probability. 
\begin{itemize}
\item As we have already mentioned,  the flip of the sign of the velocity  in case of rejection of the proposal move guarantees time-reversibility. As a consequence the proposal moves are symmetric  and the acceptance  probability can be defined only in terms of the energy difference.
\item We are slightly abusing notation in going from the original $H(q,p)$ to $H(q,v)$. However notice that $\Hf (q,v)$ is preserved by the flow \eqref{eqn:Hsplit}. 
\item The relevant energy difference here is $\Hf (q,v')- \Hf (q_*, v_*)$ (rather than $\Hf (q,v)-\Hf (q_*, v_*)$);  indeed the first step in the definition of the  proposal $(q_*, v_*)$, namely the OU process $\Theta\de_0(q,v)$, is based on an exact 
integration and preserves the desired invariant measure. Therefore the accept-reject mechanism (which is here only to  preserve the overall reversibility of the chain by accounting for  the numerical error made by the integrator $\chi_{\tau}^h$) doesn't need to include also the energy difference $\Hf (q,v)-\Hf (q, v')$. 
\item The Hamiltonian $\Hf (q,v)$, defined in \eqref{haminf}, is almost surely infinite in an infinite dimensional context; this can be seen by just applying a zero-one law to the series representation of the scalar product $\langle q, \C^{-1} q \rangle$.   However, in order for the acceptance probability to be well defined, all we need is for the  difference  
$\Hf (q, v')-\Hf (q_*, v_*)$ to be almost surely finite, i.e. for  
$\Delta \Hf  (q, v')$ to be a bounded operator. This is here the case thanks to the choice of the Verlet algorithm. Indeed from  \cite[page 2212]{BPSS11} we know that
\begin{align}
\Delta H(q, v')&=\Psi(q)-\Psi(q_*) -\frac{\delta}{2}
\left( \langle \nabla \Psi(q), v' \rangle + \langle \nabla \Psi(q_*), v_* \rangle\right) \nonumber\\
&+\frac{\delta^2}{8}\left( \|\cC^{1/2}\nabla\Psi (q_*)\|^2
 - \|\cC^{1/2}\nabla\Psi (q)\|^2 \right).  \nonumber
\end{align}
More details on this fact can be found in \cite[page 2210, 2212, 2227]{BPSS11}.
\end{itemize}
}
\end{remark}
\begin{lemma}\label{lemma:alpha-1}
Let Assumption \ref{ass:1} hold. Then, for any $p\geq 1$, 
\begin{align}\label{alpha-1}
\EE_x\lv 1-\alpha\de \rv^p \les \delta^{2p} (1+\nor{q}^{4p}+\nor{v}^{4p}).
\end{align}
\end{lemma}
\begin{proof}
The proof of Lemma \ref{lemma:alpha-1} can be found in Appendix B. 
\end{proof}
 The above \eqref{alpha-1} quantifies the intuition that the acceptance rate is very high, i.e. the proposal is rejected very rarely. Therefore the analysis of Section \ref{sec:AnalysisDrift} and Section \ref{sec:analysisNoise} is done by bearing in mind that 
``everything goes as if $\alpha\de$ were equal to one".  We now state a few technical results, gathered in Lemma \ref{psiq-psiq*} and Lemma \ref{lemma:q-q*},   that will be frequently  used in the following.

\begin{lemma}\label{psiq-psiq*}
Let Assumption \ref{ass:1} hold. Then, 
 for any $q,\tilde{q},v,\tilde{v} \in \cH^s,$
\begin{align}
 & \|\cC \nabla \Psi(q)-\cC \nabla\Psi(\tilde{q})\|_s \lesssim \| q-\tilde{q}\|_s  \mbox{ and } \nor {\cC \nabla \Psi(q)}\les (1+\nor{q});\label{lemma6.1i}\\
 & \lv\langle \nabla\Psi(q),v \rangle \rv \lesssim  (1+\| q \|_s) \|v\|_s;  \nonumber\\
& \lv \langle \nabla \Psi(q), v\rangle - \langle \nabla \Psi(\tilde{q}), \tilde{v}\rangle\rv
 \lesssim  \nor{v}\nor{q-\tilde{q}}+ (1+ \nor{\tilde{q}}) \nor{v-\tilde{v}};   \label{lemma6.1iii}  \\
& \|\cC^{1/2}\nabla\Psi(q)\|  \lesssim  1+ \nor{q};  \nonumber  \\
& \| \cC^{1/2}\nabla\Psi(q) -\cC^{1/2}\nabla\Psi(\tilde{q})\|
 \lesssim  \nor{q-\tilde{q}}.  \label{lemma6.1v}
\end{align}

\begin{proof} See \cite[Lemma 4.1]{BPSS11}
\end{proof}
\end{lemma}

Recall that $B_2(t)$ is an  $\mathcal{H}^s$-valued  $\cC_s$-Brownian motion and  that    $\xi\de$
is the noise component of the OU process $\Theta_0\de$, i.e.
\begin{align}\label{v'}
v'=e^{- \delta \Gamma_2 }v+\int_0\de e^{- (\delta-u) \Gamma_2}\sqrt{2\Gamma_2}dB_2(u)=:e^{- \delta \Gamma_2 }v+ \xi\de.
\end{align}
By integrating $\chi\de$ and $\Theta_0\de$,  the proposal move at step $k$, 
$x_*\kpod=(q_*\kpod, v\kpod_*)$, is given by
\begin{align}\label{eqn:q*v*}
& q\kpod_*= \cos\delta q\kd + \sin \delta \left( v\kd\right)' - \frac{\delta}{2}\sin{\delta}\, \cC \nabla \Psi (q\kd),\\
& v\kpod_*=-\sin\delta q\kd + \cos\delta\left( v\kd\right)' -\frac{\delta}{2}\cos\delta \,  \cC \nabla \Psi (q\kd)
- \frac{\delta}{2}\, \cC \nabla \Psi (q\kpod_*).
\end{align}
If $\gamma^k:=\gamma\kd\sim $ Bernoulli$(\alpha\kd)$, then the $(k+1)^\mathrm{th}$ step of the Markov chain is
\begin{align}\label{step}
& q\kpod=\gamma^k q_*\kpod+(1-\gamma^k)q\kd,             \nonumber\\
& v\kpod=\gamma^k v_*\kpod-(1-\gamma^k) (v\kd)' .
\end{align}

\begin{lemma}\label{lemma:q-q*}
Let Assumption \ref{ass:1} hold. Then, for any $p\geq 1$, we have
\begin{align}
& \EE\|\xi\de\|_s^p\lesssim \delta^{p/2}; \label{lemma6.2i}\\
& \EE \|(v\kd)'\vert x\kd=x\|_s^p\lesssim 1+\|v\|_s^p; \label{lemma6.2ii}\\ 
& \EE [\|q_*\kpod-q\kd\|_s^p\vert x\kd=x]\lesssim \delta^p (1+\nor{q}^p+\nor{v}^p). \label{lemma6.2iii} 
\end{align}
\end{lemma}
\begin{proof} See Appendix B. 
\end{proof}



\subsection{Analysis of the drift}\label{sec:AnalysisDrift}
Let $G(x)$ be the map in \eqref{GGGGG} with $\Gamma_1=0$, i.e.
$$
G(x)=G(q,v)=\left[\begin{array}{c}
v \\
-q -\cC \nabla \Psi(q)-\Gamma_2 v
\end{array} 
\right],
$$
and $G_i(x)$  and  $G_i\de, i =1,2, $ be the $i^\mathrm{th}$ component of $G$ and $G\de$, respectively.

\begin{lemma}\label{lemma:d-G}
Let Assumption \ref{ass:1} hold. Then, for any $x=(q,v) \in \cH^s \times \cH^s$,
\begin{align}\label{d-G1}
\|G\de_1(x)-G_1(x)\|_{s}\les \delta\, (1+ \|q\|_{s}^6+ \nor{v}^6),
\end{align}
\be 
\|G\de_2(x)-G_2(x)\|_{s}\les \delta\, (1+ \|q\|_{s}^6+ \nor{v}^6).
\ee
\end{lemma}
\begin{proof}
By \eqref{step}, 
\begin{align}\label{blabla}
& q\kpod -q\kd= \gamma^k (q_*\kpod-q\kd)\nonumber,\\
 & v\kpod -v\kd = \gamma^k v_*\kpod +(\gamma^k -1) (v\kd)'-v\kd.
\end{align}
So if we define
\begin{align*}
&A_1:=\frac{1}{\delta}\nor{\EE_x \left[\gamma\de (\cos\delta -1)q \right]}, \\
&A_2:= \nor{\EE_x \left(\gamma\de \frac{\sin\delta}{\delta}e^{- \delta \Gamma_2 }v, \right)-v} \\
&A_3:=  \lv\!\lv\EE_x \left[ \gamma\de \frac{ \sin\delta}{\delta}\xi\de -\gamma\de \frac{\sin\delta}{2}\cC \nabla\Psi(q) 
 \right]\rv\!\rv_s\\
\end{align*}
 and 
\begin{align*}
&E_1:= \lv\!\lv  q - \EE_x\left( \gamma\de \frac{\sin\delta}{\delta}q \right)\rv\!\rv_s, \\
&E_2:= \lv\!\lv  \cC\nabla\Psi(q) +\EE_x \left( -\frac{\gamma\de}{2}\cos\delta \cC\nabla \Psi(q)
 -\frac{\gamma\de}{2}\cC\nabla \Psi(q_*\kpod) \right) \rv\!\rv_s, \\
&E_3:= \lv\!\lv \EE_x\left( \gamma\de \frac{\cos\delta}{\delta}e^{- \delta \Gamma_2} v\right) -\frac{1}{\delta} v +\Gamma_2 v 
+\EE_x \left[ \frac{\gamma\de -1}{\delta}e^{- \delta \Gamma_2} v \right] \rv\!\rv_s, \\
&E_4:= \frac{1}{\delta}\lv\!\lv  \EE_x\left[ \gamma\de \cos\delta \,\xi\de + \left( \gamma\de -1 \right) \xi\de \right]  \rv\!\rv_s ,\\
\end{align*}
by the definition of $G\de$ (equation \eqref{eqn:ddelta}) and using \eqref{v'} and \eqref{eqn:q*v*},  we obtain
$$
\|G\de_1(x)-G_1(x)\|_{s}\leq A_1+A_2+A_3 \quad \mbox{and} \quad \|G\de_2(x)-G_2(x)\|_{s}
\leq E_1+E_2+E_3 +E_4. 
$$
We will bound the $A_i$'s and the $E_i$'s one by one. To this end, we will repeatedly use the following simple bounds:
\begin{align}
&\gamma\de, \gamma^k \in \{0,1\} \quad \mbox{and}  \quad 0\leq \alpha\de\leq 1 \label{fact1}; \\
& \EE[\xi \de]=0; \label{fact2}\\
& \nor{\EE[(\alpha\de-1) \xi\de]}\leq \left[\EE(\alpha\de-1)^2\right]^{1/2}\left [\EE\nor{\xi\de}^2\right]^{1/2}
\les \delta^{5/2}(1+\nor{q}^4+\nor{v}^4). \label{fact3}
\end{align}
\eqref{fact3} follows from using   Bochner's inequality
 \footnote{Let $(X,\| \cdot \|)$ be a Banach space and
 $f\in L^1((\Omega, \mathcal{F}, \mu); X)$. Then $ \|\int  f d \mu \| \leq \int \|  f\|  d \mu $.  For a proof of the Bochner's inequality 
see \cite{RoecknerPrevot07}.
}
 and Cauchy-Schwartz first  and then \eqref{lemma6.2i} and \eqref{alpha-1}. Using \eqref{fact1},  it is straightforward to see that
$$
A_1 \leq \delta \nor{q}.
$$
As for $A_2$, 
\begin{align*}
A_2 & =  \lv \! \lv  \left( I-\EE_x (\alpha\de) \frac{\sin\delta}{\delta}e^{- \delta \Gamma_2 }\right) v  \rv\!\rv_s\\
 &\leq  \lv 1-\EE_x (\alpha\de)   \rv  \nor{v} + \lv \! \lv \EE_x (\alpha\de)  \left(  1-  \frac{\sin\delta}{\delta}
e^{- \delta \Gamma_2 }\right)v\rv\!\rv_s \\
     & \leq  \delta^2 (1+\nor{q}^4+\nor{v}^4)\nor{v}+\delta \nor{v}\leq \delta (1+\nor{q}^6+\nor{v}^6),
\end{align*}
having used, in the second inequality, \eqref{alpha-1} and \eqref{fact1}. 
 $A_3$ is bounded by using \eqref{fact2}, \eqref{fact3} and \eqref{lemma6.1i}:
\begin{align*}
A_3 & \leq \lv \! \lv \frac{\sin\delta}{\delta} \EE_x \left[ (\alpha\de -1)\xi\de +\xi\de \right]      \rv\!\rv_s+
 \lv \! \lv    \EE_x \delta \,\cC \nabla\Psi(q) \rv\!\rv_s \\
& \les  \delta^{5/2} (1+\nor{q}^4+\nor{v}^4)+\delta (1+\nor{q})\leq \delta  (1+\nor{q}^4+\nor{v}^4).
\end{align*}
Hence \eqref{d-G1} has been proven. We now come to estimating the $E_i$'s. Proceeding as in the bound for $A_2$ above we obtain:
\begin{align*}
E_1&\leq \lv \! \lv q-\EE_x(\alpha\de) q    \rv\!\rv_s  +\lv \! \lv  \EE_x(\alpha\de)\left( 1-\frac{\sin\delta}{\delta} \right)q   \rv\!\rv_s\\
& \leq \delta^2  (1+\nor{q}^4+\nor{v}^4)\nor{q}+ \delta^2 \nor{q}\leq \delta^2  (1+\nor{q}^6+\nor{v}^6).
\end{align*}
Also, 
\be
E_2& \leq  \lv \! \lv \cC\nabla\Psi(q) -\EE_x( \alpha\de) \cos\delta \, \cC\nabla\Psi(q)  \rv\!\rv_s + \lv \! \lv  \frac{1}{2} \EE_x 
\left[\alpha\de \cos\delta\, \cC \nabla\Psi(q)  -\alpha\de \cC \nabla\Psi(q_*\kpod)\right]\rv\!\rv_s\\
& \les \lv \! \lv ( 1- \EE_x( \alpha\de) )   \cC\nabla\Psi(q)    \rv\!\rv_s
+ \nor{(\cos\delta-1)\cC\nabla\Psi(q)}+ \lv \! \lv    \EE_x  \left(  \cC\nabla\Psi(q) - \cC \nabla\Psi(q_*\kpod) \right)   \rv\!\rv_s\\
& \les \delta^2   (1+\nor{q}^6+\nor{v}^6) +\delta\EE_x \nor{q_*\kpod -q} \stackrel{\eqref{lemma6.2iii}}{\les} \delta (1+\nor{q}^6+\nor{v}^6) ,
\ee
where the penultimate inequality is obtained by using \eqref{alpha-1} and \eqref{lemma6.1i}.

For the last two terms:
\be
E_3  &\leq \frac{1}{\delta} \lv \! \lv \EE_x (\alpha\de)(\cos\delta -1) e^{-  \delta \Gamma_2 } v   \rv\!\rv_s
+ \frac{1}{\delta} \lv \! \lv \EE_x (\alpha\de -1)  e^{- \delta \Gamma_2 } v  \rv\!\rv_s
+  \frac{1}{\delta} \lv \! \lv  \EE_x  \left(  e^{- \delta \Gamma_2 } -1 +\delta\Gamma_2  \right) v \rv\!\rv_s\\
& \stackrel{\eqref{fact1}}{\les} \delta \nor{v}+\frac{1}{\delta} \EE \lv \alpha\de -1 \rv \nor{v}\\
& \stackrel{\eqref{alpha-1}}{\les} \delta \nor{v}+\delta  (1+\nor{q}^4+\nor{v}^4) \nor{v}\les \delta  (1+\nor{q}^6+\nor{v}^6).
\ee
Finally, from \eqref{fact2} and \eqref{fact3},
\begin{align*}
E_4 &\leq   \frac{1}{\delta}   \lv \! \lv  \EE_x (\alpha\de \cos\delta \,\xi\de)  \rv\!\rv_s
+  \frac{1}{\delta}   \lv \! \lv  \EE_x (\alpha\de -1) \xi\de  \rv\!\rv_s \\
& \les \frac{1}{\delta} \nor{\EE_x\left[ (\alpha\de -1)\cos\delta \xi\de +\cos\delta\xi\de  \right]   }+
\frac{1}{\delta} \nor{\EE_x \left[(\alpha\de -1) \xi\de\right]}
\leq \delta^{3/2} 
 (1+\nor{q}^4+\nor{v}^4) .
\end{align*}
This concludes the proof.
\end{proof}
Let us now show that condition \eqref{sizeincrem} is satisfied as well. 
\begin{lemma}\label{Lemmasize}
Under Assumptions \ref{ass:1},  the chain $x\kd\in \mathcal{H}^s \times \mathcal{H}^s$ defined in \eqref{eqn:algK} satisfies
\begin{align}\label{lemmasize1}
\EE \left[ \nor{q\kpod- q\kd}\vert x\kd=x\right]\les \delta \left( 1+ \nor{q}+\nor{v} \right),
\end{align}
\begin{align}\label{lemmasize2}
\EE \left[ \nor{v\kpod- v\kd}\vert x\kd =x \right]\les \delta^{1/2} \left( 1+ \nor{q}^6+\nor{v}^6 \right).
\end{align}
In particular, \eqref{sizeincrem} holds with $r=1/2$ and $n=6$.
\end{lemma}
\begin{proof}
 \eqref{lemmasize1} is a straightforward consequence of \eqref{blabla}, \eqref{fact1} and \eqref{lemma6.2iii}. In order to prove 
\eqref{lemmasize2} we start from \eqref{blabla} and we write 
\begin{align}\label{bit}
\EE \left[ \nor{v\kpod- v\kd}\vert x\kd=x\right]&=\EE \nor{\gamma\de v_* + (\gamma\de -1) v' -v}\\
& \les \EE\nor{\gamma\de \left( \sin\delta q + \delta \cC D\Psi (q)+\delta \cC D\Psi (q_*)\right)} \\
& + \EE \nor{\gamma\de \cos\delta v'- (1-\gamma\de)v' - v }.
\end{align}
By using \eqref{fact1}, \eqref{lemma6.1i} and \eqref{lemma6.2iii} we get
\begin{align}\label{bit1}
\EE\nor{\gamma\de \left( \sin\delta q + \delta \cC D\Psi (q)+\delta \cC D\Psi (q_*)\right)}\les \delta ( 1+ \nor{q}).
\end{align}
Notice that
\begin{align}\label{ops}
\EE \lv \gamma\de -1 \rv^{\ell} = 1-\EE (\alpha\de), \qquad \forall \ell\geq 1.
\end{align}
Therefore by  \eqref{v'} and  \eqref{fact1} and repeatedly using \eqref{ops}, 
\begin{align}
\EE \nor{\gamma\de \cos\delta v'- (1-\gamma\de)v' - v }&\les \EE \nor{\left[ \gamma\de \cos\delta - (1-\gamma\de) \right] \xi\de} \nonumber \\
& + \EE \nor{(1-\gamma\de) e^{-\delta \Gamma_2}v}+
\EE \nor{\gamma\de \cos\delta e^{-\delta \Gamma_2}v - v  } \nonumber \\
& \stackrel{\eqref{lemma6.2i}}{\les} \delta^{1/2} + \EE\lv 1-\alpha\de \rv \nor{v} \nonumber \\
&+\EE \lv 1-\gamma\de \cos \delta \rv \nor{ e^{-\delta \Gamma_2}v}  + \nor{ e^{-\delta \Gamma_2}v -v} \nonumber \\
& \stackrel{\eqref{alpha-1}}{\les} \delta^{1/2}+ \delta^2 (1+\nor{q}^4+\nor{v}^4) \nor{v} \nonumber \\
& + \EE\lv \gamma\de -1\rv\nor{e^{-\delta \Gamma_2}v}+
\EE\lv \gamma\de (\cos\delta -1) \rv \nor{e^{-\delta \Gamma_2}v}+ \delta \nor{v} \nonumber \\
& \les \delta^{1/2} (1+\nor{q}^6+\nor{v}^6)\label{bit2}.
\end{align}
Now \eqref{bit}, \eqref{bit1} and \eqref{bit2} imply \eqref{lemmasize2}.
\end{proof}

Finally, the a priori bound \eqref{dde} holds.
\begin{lemma}\label{lemma:sup}
Let Assumption \ref{ass:1} hold. Then the chain \eqref{eqn:algK} satisfies 
\begin{align}\label{wwp}
\sup_{\delta\in (0,1/2)}\left\{ \delta \EE \left[  \sum_{k\delta <T} \nors{x\kd}^{\ell}\right] \right\} < \infty \quad 
\mbox{for any integer }\ell\geq 1.
\end{align}
In particular, the bound \eqref{dde} holds  (with $\epsilon=1$ and for any moment of $\nors{x\kd}$).
\end{lemma}
\begin{remark}\label{rem:flip} \textup{
Before proving the above lemma, let us make some comments. First of all, the estimate of condition \eqref{dde} is needed mainly because the process has not been started in stationarity and hence  it is not stationary. For the same reason an analogous estimate was needed in  \cite{PST13}, as well. However there  the approximate drift grows linearly in $x$ (see \cite[Equation (25)]{PST13}) and this is sufficient to prove an estimate of the type \eqref{dde}.  Here, because in case of rejection of the proposed move the sign of the velocity is flipped, the approximate drift grows faster than linearly (see Lemma \ref{lemma:d-G} and \eqref{nonldrift}).  To deal with the change of sign of the velocity we will observe that such a change of sign doesn't matter if we look at even powers of $x\kd$ -- what matters is that in moving from $x\kd$ to $x\kpod$ we always ``move a short distance''--  and we will exploit the  independence of $v\kd$ and $\xi\de$, once $x\kd$ is given. }
\end{remark}
\begin{proof} If we show that \eqref{wwp} is true for every even $\ell$  then it is true for every $\ell\geq 1$. 
Indeed
$$
\nors{x}^{\ell}\leq \nors{x}^{2\ell}+1 \qquad \mbox{so} \qquad \delta \sum_{k\delta <T}
\nors{x}^{\ell} \les \delta \sum_{k\delta <T} \nors{x}^{2\ell}+1< \infty. 
$$
 Throughout this proof $c$ will be a generic positive constant.
We begin by recalling the definition of  the map $\Theta_1^{\delta/2}$:
$$
\Theta_1^{\delta/2}(q,p)=\left(q, v-\frac{\delta}{2} \cC \nabla \Psi(q)\right),
$$
hence
\begin{align*}
\nors{\Theta_1^{\delta/2}(x)}^2&= \nor{q}^2+\nor{v}^2+\frac{\delta^2}{4} \nor{\cC \nabla\Psi(q)}^2-\delta \langle v, \cC\nabla\Psi(q)\rangle\\
& \stackrel{\eqref{lemma6.1i}}{\leq}  \nor{q}^2+\nor{v}^2+ c \, \delta^2 ( 1+\nor{q}^2) + c\,\delta \nor{v}(1+\nor{q})\\
& \leq (1+c\, \delta)  \nors{x}^2+ c\,\delta.
\end{align*}
Because $R^{\delta}$ is a rotation, it preserves the norm, so also
\begin{align}\label{(B)}
\nors{\chi^{\delta}(x)}^2=\nors{\Theta^{\delta/2}_1\circ R^{\delta}\circ \Theta^{\delta/2}_1(x)}^2 \leq (1+c\, \delta)  \nors{x}^2+ c\,\delta.
\end{align}
Now notice that by definition $x\kpod=(q\kpod,v\kpod)$ is either equal to $\chi^{\delta}(q\kd, (v\kd)')=(q_*\kpod, v_*\kpod)$ (if the proposal is accepted) or to
$(q\kd, -(v\kd)')$ (if the proposal is rejected).  Thanks to \eqref{(B)}, in any of these two cases we have
$$
\nors{x\kpod}^2\leq (1+c\delta) \nors{(x\kd)'}^2+c\delta,
$$
where 
$
(x\kd)'=((q\kd)', (v\kd)')=(q\kd, (v\kd)')$. By \eqref{v'}, 
\be
\nors{(x\kd)'}^2 & \leq \nor{q\kd}^2+\nor{v\kd}^2+\nor{\xi\de}^2+2\langle e^{- \delta \Gamma_2}v\kd,\xi\de\rangle \nonumber\\
&= \nors{x\kd}^2+\nor{\xi\de}^2+ 2 \langle e^{- \delta \Gamma_2}v\kd,\xi\de\rangle.\nonumber
\ee
Therefore
\begin{align*}
\EE\nors{x\kpod}^2 & = \EE\{\EE [\nors{x\kpod}^2\vert x\kd] \}\\
& \leq (1+c\,\delta) \EE\nor{x\kd}^2 + (1+c\delta)\, \EE \nor{\xi^{\delta}}^2 \\
& + (1+c\, \delta) \EE\{ \EE[\langle e^{- \delta \Gamma_2 }v\kd , \xi^{\delta} \rangle\vert x\kd] \}+ c\, \delta.
\end{align*}
By  the conditional independence of  $v\kd$ and $\xi^{\delta}$ together with \eqref{fact2}
$$
\EE\{ \EE[\langle e^{- \delta \Gamma_2 }v\kd , \xi^{\delta} \rangle\vert x\kd] \}=0;
$$
hence, using \eqref{lemma6.2i},  we obtain
$$
\EE \nors{x\kpod}^2\leq (1+c \, \delta) \EE \nors{x\kd}^2 + c\, \delta.
$$
Iterating the above inequality  leads to 
$$
\EE \nors{x\kpod}^2\leq (1+c \, \delta) ^{\left[ T/\delta\right]} \EE \nors{x^0}^2 + c\, \delta (1+c \, \delta) ^{\left[ T/\delta\right]} +c \, \delta, 
$$
which implies
$$
\delta \sum_{k\delta< T}\EE \nors{x\kd}^2 < \infty.
$$
We now need to show that for any $j>1$, 
$$
\delta \sum_{k\delta< T}\EE \nors{x\kd}^{2j} < \infty.
$$
By the same reasoning as before we start with observing that
\begin{align*}
\nors{\chi\de(x)}^{2j} & \leq (1+c\, \delta) \nors{x}^{2j} + c\, \delta + 2 \sum_{l=1}^{j-1} (1+c\, \delta)^{l} 
\nors{x}^{2l} \delta^{j-l}\\
& \leq (1+c\, \delta) \nors{x}^{2j} + c\, \delta
\end{align*}
 (notice  that in the above $j-l\geq 1$ because $1\leq l \leq j-1 $). Hence
$$
\EE\nors{x\kpod}^{2j}\leq (1+c\, \delta) \EE \nors{(x\kd)')}^{2j}+c\delta.
$$
From \eqref{v'} we have
\begin{align*}
\nors{(x\kd)'}^{2j}&\leq \nors{x\kd}^{2j}+ \nors{\xi\de}^{2j} +
 c \left( \langle e^{- \delta \Gamma_2 }v\kd, \xi\de \rangle \right)^j\\
&+c \sum_{l=1}^{j-1}\nors{x\kd}^{2l} \nor{\xi\de}^{2(j-l)}+ c\sum_{l=1}^{j-1}
\nors{x\kd}^{2l} \left( \langle e^{- \delta \Gamma_2 }v\kd, \xi\de \rangle \right)^{j-l}\\
&+ c \sum_{l=1}^{j-1}\nor{\xi\de}^{2l}\left( \langle e^{- \delta \Gamma_2 }v\kd, \xi\de \rangle \right)^{j-l}.
\end{align*}
Using again  the conditional independence of  $v\kd$ and $\xi\de$, for any $l>1$, 
\begin{align*}
\EE \left\{  \EE\left[ \left(\langle e^{- \delta \Gamma_2 }v\kd, \xi\de \rangle \right)^l \vert x\kd =x\right]  \right\}
&\leq \EE \left\{ \EE\left[  \nor{v\kd}^l \nor{\xi\de}^l \vert x\kd=x \right]  \right\}\\
& \leq c\, \delta^{l/2} \,\EE\nor{v\kd}^l\leq c\, \delta \,\EE\nor{v\kd}^l.
\end{align*}
Therefore, 
$$
\EE \nors{(x\kd)'}^{2j} \leq \EE \nors{x\kd}^{2j}+\delta^j+\delta\EE(1+\nors{x\kd}^{2j})
$$
hence
$$
\EE \nors{x\kpod}^{2j} \leq  (1+c\, \delta)  \EE\nors{x\kd}^{2j}+c\,\delta
$$
and we can conclude as before. 
\end{proof}
\subsection{Analysis of the noise}\label{sec:analysisNoise}
Let us start with defining 
\begin{align}\label{crop}
D\de(x):=\EE \left[ M\kd \otimes M\kd \vert x\kd =x \right].
\end{align}
This section is devoted to proving the invariance principle Lemma \ref{lemmanoise} below, as a consequence of the following
 Lemma \ref{lemman1} and Lemma \ref{lemman2}, which we prove in Appendix B.  In order to state such lemmata, consider the following set of conditions:
\begin{cond}\label{cond2} 
The Markov chain $x\kd\in \mathcal{H}^s\times \mathcal{H}^s$ defined in \eqref{eqn:algK} satisfies:
\begin{description} 
\item[(i)]  There exist two integers $d_1,d_2 \geq 1$ and two real numbers $b_1,b_2 >0$ such that  
\begin{align}
& \lv \langle \hat{\varphi}_j^{\ell}, D\de(x)\, \hat{\varphi}_i^{\bar{\ell}} \rangle_{s\times s} - 
 \langle \hat{\varphi}_j^{\ell}, \mathfrak{C}_s\, \hat{\varphi}_i^{\bar{\ell}} \rangle_{s\times s} \rv \les \delta^{b_1} (1+ \nors{x}^{d_1})
\qquad \forall i,j \in \mathbb{N} \,\,\mbox{ and } \,\, \ell,\bar{\ell}\in {1,2};
 \label{lemmanoise1gen} \\
&  \lv       \textup{Trace}_{\mathcal{H}^s\times \mathcal{H}^s}(D\de(x) )- 
 \textup{Trace}_{\mathcal{H}^s\times \mathcal{H}^s}(\mathfrak{C}_s)      \rv \les \delta^{b_2} (1+ \nors{x}^{d_2}) ,                                       \label{lemmanoise2gen} 
\end{align}
where $\mathfrak{C}_s$ is the covariance operator defined in \eqref{mfcs}. 
\item[(ii)]   There exist four real numbers $\eta_1, \eta_2, \eta_3, \eta_4$ such that 
\be
b_1+1-\eta_1>0,  \quad
 b_2+1-\eta_2>0,  \quad
 4r-\eta_3>0, \quad 4-\eta_4>0.
\ee
Moreoever, the bound
\begin{align}\label{dde22}
\sup_{\delta\in (0, 1/2)}\left\{  \delta^{\eta}\, \EE \left[ \sum_{k\delta\leq T} \| x\kd \|_{s\times s}^m \right]    \right\}< \infty
\end{align}
holds for $\eta=\min_{i=1\dots 4}\{\eta_i\} $ and  $m=\max \{d_1, d_2, 4n, 4p\}$.  In the above $ n$ and $r$ are as in Condition \ref{drift-martCond}. 
\end{description}
\begin{remark}
Because $\nors{x}^b\les \nors{x}^d+1$ for all $d\geq b$, if \eqref{dde22} holds with $m=d$, then it also hold for any $m\leq d$. 
\end{remark}

\end{cond}

\begin{lemma}\label{lemman1} If \eqref{sizeincrem} is satisfied with  $r>1/4$ then the estimates \eqref{dde-G} and \eqref{sizeincrem}  together with Conditions \ref{cond2} imply the invariance principle Lemma \ref{lemmanoise}. 
\end{lemma}
\begin{proof}
See Appendix B.
\end{proof}

\begin{lemma}\label{lemman2}
Under Assumption \ref{ass:1}, the estimates \eqref{lemmanoise1gen} and \eqref{lemmanoise2gen} hold with $b_1=b_2=1/6$ and $d_1=d_2=10$.
\end{lemma} 
\begin{proof}
See Appendix B.
\end{proof}

\begin{lemma}\label{lemmanoise}
Under Assumption \ref{ass:1},  the rescaled noise process
$ \mathcal{H}^s\times \mathcal{H}^s \ni \tilde{B}^{\delta}$ defined in \eqref{rescnoise}, converges weakly in 
$C([0,T];\mathcal{H}^s\times \mathcal{H}^s)$  to $\mathcal{H}^s\times \mathcal{H}^s \ni
B=(0, B_2)$ where $B_2$ is a $\mathcal{H}_s$-valued, mean zero  $\cC_s$ Brownian motion.
\end{lemma}
\begin{proof}[Proof of Lemma  \ref{lemmanoise}]
We use Lemma \ref{lemman1}. Thanks to Lemma \ref{lemma:d-G}, \eqref{dde-G} is satisfied with $a=1$ and $p=6$. From Lemma \ref{Lemmasize}, \eqref{sizeincrem} holds with $r=1/2$.  As for Conditions \ref{cond2}, Lemma \ref{lemman2} proves that 
Condition \ref{cond2} \textbf{(i)} holds. In view of Lemma \ref{lemma:sup}, Condition  \ref{cond2} \textbf{(ii)} is satisfied with $\eta_1=\eta_2=\eta_3=\eta_4=1$.
\end{proof}

\section{Numerics} 
\label{sec:num}
Before describing the numerical results we highlight the fact that
the function space MALA algorithm of \cite{Besk:etal:08}
is a special  case of the function space HMC algorithm of \cite{BPSS11}
which, in turn, is a special case of the \sol algorithm introduced in this
paper. All of these algorithms are designed to have dimension-independent
mixing times, and are indistinguishable from this point of view. However
we expect to see different performance in practice and our numerical
experiments are aimed at demonstrating this. In the paper
\cite{BPSS11} it was shown that HMC is a significant
improvement on MALA for bridge diffusions \cite{BPSS11}.
It is natural to try and show that the \sol algorithm can be more efficient than
HMC. To do this, we choose a target measure $\pi$ defined 
with respect to a reference measure $\pi_0$ which is a standard
Brownian bridge on $[0,100]$, starting and ending at $0$, and with
\[
\Psi(q)= \frac{1}{2}   \int_0^{100}   
 \, V\big(q(\tau)\big) \,d\tau 
\label{eq:psi}
\]
and $V(u)= (u^2-1)^2$.  Thus we may take as ${\mathcal H}$ 
the space $L^2\big(  (0,\,100), \mathbb{R} \big).$
The properties of measures of this type are studied in some detail 
in \cite{otto2013invariant}.  For our purposes it is relevant to note
that $\mathbb{E}^{\pi}\,q(\tau) =0 $ for  $0\le \tau \le 100$. 
This follows from the fact that the function $V$ is even and zero
boundary conditions are imposed by the Brownian bridge meaning that
the measure is invariant under $q \mapsto -q$.
The precision (inverse covariance) operator for unit Brownian bridge
on an interval $[0,T]$ is simply given by the negative of the Laplacian
with homogeneous Dirichlet boundary conditions. Furthermore samples
may be drawn simply by drawing a Brownian motion $B(t)$ and subtracting
$tB(T)/T.$

Because of these properties of $\pi$ we expect that sufficiently
long runs of MCMC methods to sample from $\pi$ should exhibit,
approximately, this zero mean property. We may use the rate
at which this occurs as a way of discriminating between the
different algorithms.
To this end we define a quantity $E(n)$, with $n$ defined  in what
follows as $n=N_d\,N_M$. Here $N_d$ denotes the number
of steps used in the deterministic integration; in MALA,
$N_d=1$ and for our implementation of HMC, $N_d>1$ is 
chosen so that such that $\tau=N_d\, h \approx 1$.
The integer $N_M$ is the number of MCMC steps taken. 
The quantity $n$ is thus a measure of the total number of
numerical integration steps used and thus of the overall work
required by the algorithm, noting that all the algorithms 
involve more or less the same calculations per step, and that
accept/reject overheads are minimal compared with the cost
arising from accumulation of numerical integration steps. 
We define the running average
\[
\bar{q}^{\,n} (\tau) = \frac{1}{N_M} \sum_{i=1}^{N_M}  q^{i} (\tau)
\]
where the index $i$ runs over the realizations $N_M$ of the path $q$.
We then define the quantity $E(n)$ as
\[
E(n)= \frac{1}{100} \int_0^{100} \!\!\!\!  \ \big|\bar{q}^{\,n} (\tau) \big| 
d\tau.
\label{eq:Edef}
\]
When viewed as a function of $n$ the rate at which $E(n)$ approaches zero 
determines  the efficiency of the sampling. The faster $E$ decreases, 
the more efficient the sampling.
All our numerical experiments are conducted with the form of the
\sol algorithm described in Remarks \ref{rem:need}. Thus $\Gamma_2=I$ and
we use the parameter $\iota$ to implicitly define $\delta.$

For the first set of numerical experiments we use the \sol algorithm
in the form which gives rise to the diffusion limit, namely with $\delta=h$
so that we make only  one step in the deterministic integration $N_d=1.$ 
The key parameter is thus $\iota$ (and implicitly $\delta$) given in
\eqref{eq:preserve}.
We consider the values $\iota= \,0.9, \,0.99,$ and $0.999$. 
The results are summarized in Figure \ref{fig:one}.
They demonstrate that SOL-HMC is indeed considerably better than MALA, 
but is not better than HMC for this problem and this choice of
parameters.
For $\iota=0.999$, we see a plateauing of the value for $E$ for $n$ between
$500$ and $1000$.
It seems that such behavior is due to the sign-flip step when the proposal is rejected.
As Horowitz \cite{H91} noted for the standard finite dimensional  algorithm L2MC
that we have generalized,
``If it were not for this momenta reversal, it would be a near certain conclusion that L2MC is more efficient than HMC".
Our findings are consistent with this remark.
\begin{figure*}[th]
\begin{center}
\includegraphics[height=10cm]{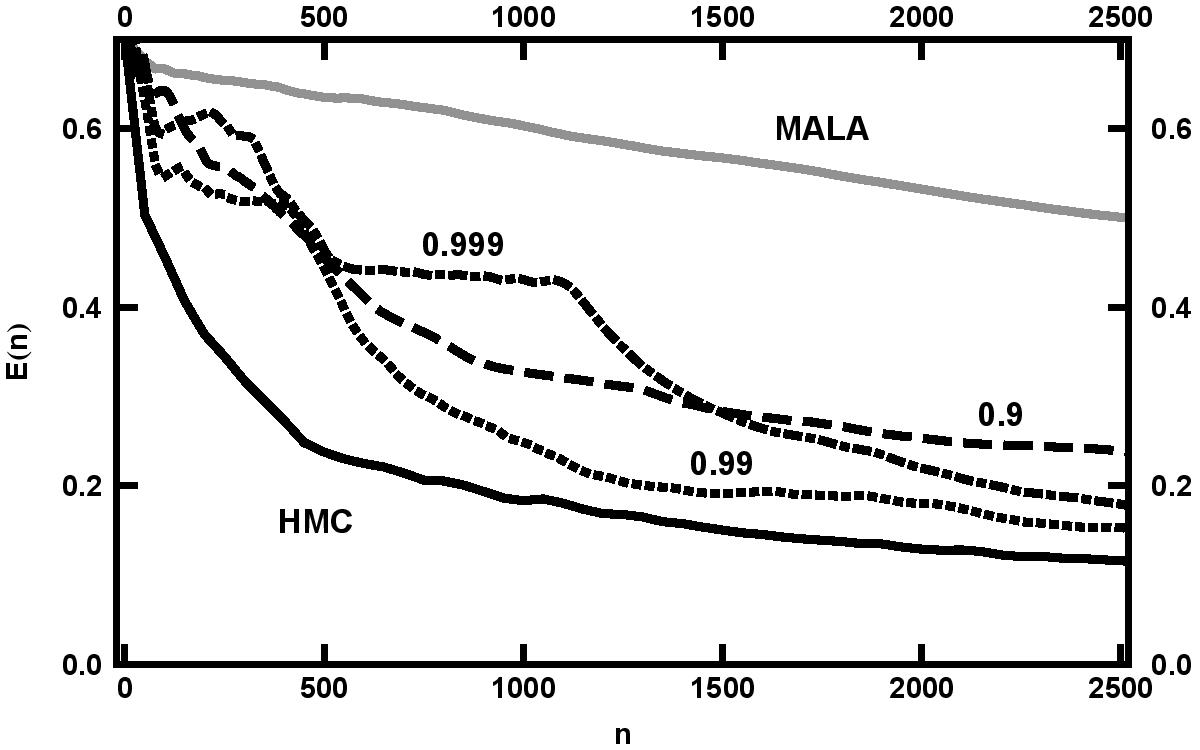}
\end{center}
\caption{The quantity $E$ plotted as a function of number of iterations. The HMC 
algorithm has the best performance; MALA has the worst.  For the values of $\iota$ 
used ($0.9$, $0.99$ and $0.999$) \sol is considerably 
better than MALA, but not better than HMC.}
\label{fig:one}
\end{figure*}

The first set of experiments, in which HMC appeared more efficient than \sol,
employed a value of $\delta$ which corresponds to making small changes in the
momentum in each proposal.
For the second set of numerical experiments we relax this constraint
and take $\iota=2^{-1/2}$ in all our \sol simulations. This corresponds
to taking an equal mixture of the current momentum and an independent
draw from its equilibrium value. 
Furthermore, these experiments use more than one step in the 
deterministic integration, $N_d>1$. 
For the HMC integration, we use $N_d =50$, $\tau=1$  and of course, $\iota=0$.
The results are summarized in Figure \ref{fig:two}
where we show the behaviour of the HMC algorithm in comparison
with four choices of parameters in the \sol algorithm: i) $N_d =10$; ii) $N_d =25$; iii) $N_d=50$; and in iv) 
$N_d$ is a random value, uniformly distributed between $25$ and $75$, in
each step of the MCMC algorithm. 
We see that if $N_d \ge 25$ the \sol algorithm shows improved
behaviour in comparison with  the HMC algorithm. 
\begin{figure*}[th]
\begin{center}
\includegraphics[height=10cm]{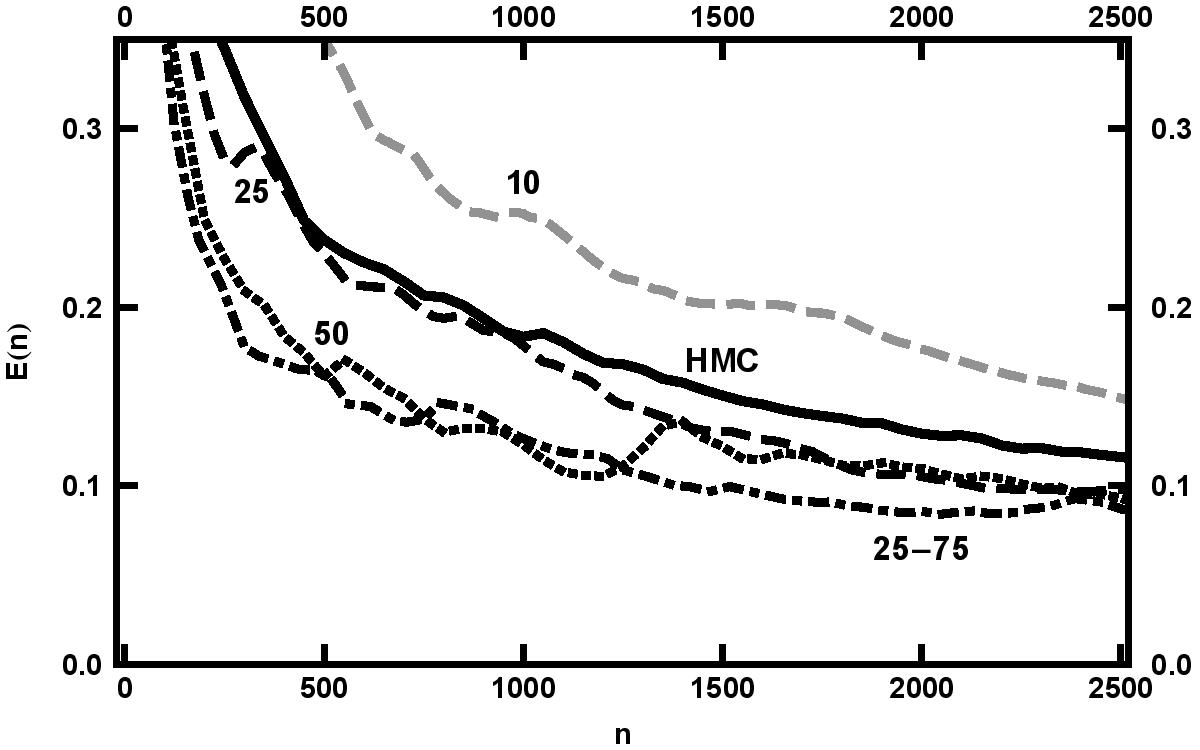}
\end{center}
\caption{Figure showing that \sol can exhibit faster mixing than
HMC when run in the regime
with $\iota= 2^{-1/2}$, for appropriate choice of $\tau$ (reflected in the
integer $N_d$ which labels the graphs -- $25$, $50$, $75$ and $25-75$ for
the random case).
Note that for HMC, by definition, $\iota=0$ and we fix $\tau \approx 1.$ Note 
also that the vertical scale is half that of the previous graph.}
\label{fig:two}
\end{figure*}

\section{Conclusions}
\label{sec:conc}
We have introduced a parametric family of MCMC methods, the
\sol algorithms,
suitable for sampling measures defined via density with respect to
a Gaussian. The parametric family includes a range of existing
methods as special parameter choices including the function space
MALA and HMC methods introduced in \cite{Besk:etal:08,BPSS11}.
Whilst both these algorithms are reversible with respect to the
target, generic parameter choices in the SOL-HMC algorithm lead
to irreversible algorithms which preserve the target.
With one particular parameter choice we show that the algorithm
has a diffusion limit to the second order Langevin equation; this latter
limit makes clear the role of irreversibility in the algorithm.
Numerical results indicate that the method is comparable with
the function space HMC method of \cite{BPSS11} which, in turn,
is superior to the function space MALA method of \cite{Besk:etal:08}. 
Indeed, in the example studied, we are able to exhibit situations for which the
\sol algorithm outperforms the HMC method. 
Further application of the method is thus suggested. 

We make an important observation about the diffusion limits proved
in this paper, Theorem \ref{thm:difflim}, and in \cite{PST13}, both of which concern algorithms
that have been specifically designed to deal with target measures
defined via density with respect to a Gaussian; indeed both methods
would suffer no rejections in the pure Gaussian case. The limit theorems
demonstrate that the number of steps required to sample may be chosen
independently of the dimension of the approximating space $N$. However,
in contrast to the diffusion limits identified in \cite{Matt:Pill:Stu:11,Pill:Stu:Thi:12} the theory does
not reveal an optimal choice for the time-step, or an optimal acceptance
probability. The fact that an optimal acceptance probability, and implicitly
an optimal time-step, can be identified in \cite{Matt:Pill:Stu:11,Pill:Stu:Thi:12} is precisely because
the proposal does not exactly preserve the underlying Gaussian reference 
measure and the universal optimal acceptance probability is determined purely 
by the Gaussian properties of the problem; the change of measure, and hence
function $\Psi$, play no role. Once improved methods are
used, such as SOL-HMC and the pCN method analyzed in \cite{PST13}, which exactly
preserve the Gaussian structure, no such universal behaviour can be
expected and optimality must be understood on a case by case basis.

\section*{Appendix A}
This Appendix contains the proof of Theorem \ref{t:alg}, of the identity  \eqref{mapB} and a sketch of the proof of Proposition \ref{lemma:difflimit}. 

\begin{proof}[Proof of Theorem \ref{t:alg}]
{By following the arguments given in Section 5.3 of
\cite{neal2010mcmc},
it suffices to show $\chi_{\tau}^h$ preserves $\Pi$ when appended
with the suitable accept-reject mechanism.} We show this
  using the fact that the finite dimensional version of the above
  algorithm preserves the corresponding 
 invariant measure. Since the proof of this
 is very similar to that of Theorem \ref{t:flowprev} 
 we only sketch the details.
For $N \in \mathbb{N}$ and $t> 0$ define the map
 \be
\chi^t_N &= \Theta^{t/2}_{N,1} \circ \sR^t_N \circ \Theta^{t/2}_{N,1}
\ee
where $R^t_N$ and $\Theta^t_N$ are obtained by restricting
$R^t$ and $\Theta^t_1$ respectively on the first $N$ components of $(q,v)$
and with $\Psi$ replaced by $\Psi^N$, {as in the proof of 
Theorem \ref{t:flowprev}.
Also following the ideas used in the proof of Theorem \ref{t:flowprev}
it may be shown that}
$\lim_{N \rightarrow \infty} \chi^t_N(q,v) =  \chi^t(q,v)$
for any $t$ and $(q,v) \in \h^s \times \h^s$. Now the proof can be
completed via a dominated convergence argument similar to that
used in the proof of Theorem \ref{t:flowprev}.
\end{proof}
\begin{proof}[Proof of \eqref{mapB}]
We  recall that for any integer $j\geq 0$
\begin{align}\label{sumk}
x^{j+1, \delta}-x^{j,\delta}=\delta G\de(x^{j,\delta})+\sqrt{2\delta\mathcal{S}}M^{j,\delta}.
\end{align}
 With this in mind \eqref{mapB} is a straightforward consequence  of the definition  of $z\de(t)$ and $\bar{z}\de(t)$, given in \eqref{eqn:continter} and 
\eqref{eqn:constinter} respectively; indeed if  $t_k\leq t < t_{k+1}$ then 
\begin{align}
z\de(t)&=\frac{1}{\delta}(t-t_k) x\kpod + \frac{1}{\delta} (t_{k+1}-t) x\kd \nonumber\\
&\stackrel{\eqref{sumk}}{=} \frac{1}{\delta} (t-t_k) \left[ x\kd+G\de(x\kd)\delta +\sqrt{2\delta\mathcal{S}} M\kd  \right]
+\frac{1}{\delta} (t_{k+1}-t)x\kd \nonumber \\
&= \frac{1}{\delta} (t_{k+1}-t_k) x\kd + (t-t_k) G\de(\bar{z}\de(t)) +\sqrt{\frac{2\mathcal{S}}{\delta}} 
(t-t_k)M\kd \nonumber \\
&= x\kd +(t-t_k) G\de(\bar{z}\de(t))+\sqrt{\frac{2\mathcal{S}}{\delta}} (t-t_k)M\kd \nonumber\\
&= x\kd +\int_{t_k}^t G\de(\bar{z}\de(u)) du + \sqrt{\frac{2\mathcal{S}}{\delta}} (t-t_k)M\kd. \label{sumk2}
\end{align}
Equation \eqref{mapB} comes from extending the above equality to the whole interval $[0,t]$. More precisely, taking the sum for 
$j=0, \dots, k-1$ on both sides of \eqref{sumk}, we obtain
\begin{align}\label{sumk3}
x\kd=x_0+\int_0^{t_k}G\de(\bar{z}\de(u))\,du+ \sqrt{2\mathcal{S}\delta}\sum_{j=0}^{k-1}M^{j,\delta}.
\end{align}
Substituting \eqref{sumk3} into \eqref{sumk2} we obtain
\be
z\de(t)&=x^0+\int_0^{t}G\de(\bar{z}\de(u))\,du+ \sqrt{2\mathcal{S}\delta} \sum_{j=0}^{k-1}M^{j,\delta}
+\sqrt{\frac{2\mathcal{S}}{\delta}} (t-t_k)M\kd\\
&= x^0+\int_0^{t}G(z\de(u)) du + \tilde{B}\de(t)+\int_0^t \left[  G\de(\bar{z}\de(u)) - G(z\de(u)) \right] du\\
&= x^0+\int_0^{t}G(z\de(u)) du + \hat{B}\de(t) = \Upsilon(x^0, \hat{B}\de).
\ee
\end{proof}
\begin{proof}[Proof of Proposition \ref{lemma:difflimit}]
Once \eqref{mapB} has been established, this proof is completely analogous to the proof of \cite[Lemma 3.5]{PST13}, which is divided in four steps. The only step that we need to specify here is the third, the rest remains unchanged. Such a step consists in showing that
$$
\lim_{\delta \rightarrow 0} \EE \left[ \int_0^T \nors{G\de(\bar{z}\de(u))-G(z\de(u))} du  \right]=0,
$$
where we recall that  $z\de(t)$ and $\bar{z}\de(t)$ have been defined in  \eqref{eqn:continter} and 
\eqref{eqn:constinter} respectively. To this end notice first that if $t_k\leq u < t_{k+1}$ then
$$
z\de(u)-\bar{z}\de (u)=\frac{1}{\delta}(u - t_k) x\kpod + \frac{1}{\delta}(t_{k+1}-u-\delta)x\kd=
\frac{1}{\delta}(u - t_k) (x\kpod - x\kd). 
$$
Hence $ \nors{z\de(u)-\bar{z}\de(u)} \leq \nors{x\kpod - x\kd}$. By using this inequality together with \eqref{dde-G} and \eqref{sizeincrem},   we have, for $t_k\leq u < t_{k+1}$, 
\be
\EE \nors{G\de(\bar{z}\de(u))-G(z\de(u))}&\leq \EE \nors{G\de(\bar{z}\de(u))-G(\bar{z}\de(u))}+
\EE \nors{G(\bar{z}\de(u))-G(z\de(u))}\\
& \les \delta^a (1+\EE\nors{\bar{z}\de(u)}^p) + \EE \nors{z\de(u)-\bar{z}\de(u)}\\
& \les \delta^a (1+\EE\nors{x\kd}^p) +\delta^r (1+\EE\nors{x\kd}^n)\leq 
\delta^{a\wedge r} (1+\EE\nors{x\kd}^{p\vee n}).
\ee
Therefore, using \eqref{dde}, 
\be
\EE \left[ \int_0^T \nors{G\de(\bar{z}\de(u))-G(z\de(u))}\, du  \right]
&\les  \delta^{1+(a\wedge r)}   \sum_{k\delta <T}\left(1+ \EE \nors{x\kd}^{p\vee n}  \right)\\
&\les \delta^{a\wedge r}+ \delta^{1-\epsilon + (a\wedge r)}\left(\delta^{\epsilon}   
  \sum_{k\delta <T} \EE\nors{x\kd}^{p\vee n}  \right) \rightarrow 0.
\ee
\end{proof}

\bigskip

\section*{Appendix B}
This appendix contains the proof of several technical estimates contained in the paper. In particular, it contains the proof of  Lemma \ref{lemma:alpha-1},  Lemma  \ref{lemma:q-q*}, Lemma \ref{lemman1} and Lemma \ref{lemman2}.
We start with an analysis of the acceptance probability. We recall that from  \cite[page 2212]{BPSS11} we know that
\begin{align}
\Delta H(q, v')&=\Psi(q)-\Psi(q_*) -\frac{\delta}{2}
\left( \langle \nabla \Psi(q), v' \rangle + \langle \nabla \Psi(q_*), v_* \rangle\right) \nonumber\\
&+\frac{\delta^2}{8}\left( \|\cC^{1/2}\nabla\Psi (q_*)\|^2
 - \|\cC^{1/2}\nabla\Psi (q)\|^2 \right)  \nonumber\\
& = \mathcal{F}_{\Psi} +\frac{\delta^2}{8}\left( \|\cC^{1/2}\nabla\Psi (q_*)\|^2
 - \|\cC^{1/2}\nabla\Psi (q)\|^2 \right),\nonumber
\end{align}
having set 
$$
\mathcal{F}_{\Psi}:= \mathcal{F}_{\Psi}(x,\xi\de)=
\Psi(q)-\Psi(q_*) -\frac{\delta}{2}
\left( \langle \nabla \Psi(q), v' \rangle + \langle \nabla \Psi(q_*), v_* \rangle\right).
$$

\begin{proof}[Proof of Lemma \ref{lemma:alpha-1}]
Let 
$$
\tilde{\alpha}\de(x,\xi\de)=\tilde{\alpha}\de:= 1\wedge \exp (\mathcal{F}_{\Psi}) \quad \mbox{and} \quad      
\bar{\alpha}\de(x,\xi\de)=\bar{\alpha}\de:=1+\mathcal{F}_{\Psi}{\bf 1}_{\{\mathcal{F}_{\Psi}\leq 0\}}. 
$$
Introducing the functions $h,\bar{h}:\bbR\rightarrow \bbR$ defined as
$$
h(y)=1\wedge e^y \quad \mbox{and} \quad \bar{h}(y)= 1+y {\bf 1}_{\{y\leq 0\}}
$$
we have
\begin{align}
\alpha\de &= h\left(\mathcal{F}_{\Psi}+ \frac{\delta^2}{8}\left( \|\cC^{1/2}\nabla\Psi (q_*)\|^2
 - \|\cC^{1/2}\nabla\Psi (q)\|^2 \right)\right),\label{adh}\\
\tilde{\alpha}\de &= h\left(\mathcal{F}_{\Psi}\right)  \quad \mbox{and} \quad
\bar{\alpha}\de = \bar{h}\left( \mathcal{F}_{\Psi}\right).\label{atbdh}
\end{align}
Clearly,
$$
\EE_x\lv 1-\alpha\de \rv^p \les \EE_x\lv \alpha\de - \tilde{\alpha}\de \rv^p + \EE_x\lv \tilde{\alpha}\de- \bar{\alpha}\de \rv^p
+ \EE_x\lv \bar{\alpha}\de - 1 \rv^p. 
$$
We will show that
\begin{align}\label{alpha1bit}
 \EE_x\lv \alpha\de - \tilde{\alpha}\de \rv^p \les \delta^{3p} (1+\nor{q}^{2p}+ \nor{v}^{2p})
\end{align}
\begin{align}\label{alpha2bit}
\EE_x\lv \tilde{\alpha}\de- \bar{\alpha}\de \rv^p\les \EE_x\lv \mathcal{F}_{\Psi}\rv^{2p}
\end{align}
\begin{align}\label{alpha3bit}
\EE_x\lv \bar{\alpha}\de - 1 \rv^p \les \EE_x\lv \mathcal{F}_{\Psi}\rv^{p}
\end{align}
The above three bounds, together with 
\begin{align}\label{boundF}
\EE_x\lv \mathcal{F}_{\Psi}\rv^{p}\les \delta^{2p} (1+\nor{q}^{2p}+ \nor{v}^{2p}),  
\end{align}
imply \eqref{alpha-1}. Let us start with \eqref{alpha1bit}: from \eqref{adh}  and \eqref{atbdh} and using the Lipshitzianity of the function  $h$ we have
\begin{align*}
 \EE_x\lv \alpha\de - \tilde{\alpha}\de \rv^p &= \EE_x \lv  h\left(\mathcal{F}_{\Psi}+ \frac{\delta^2}{8}\left(
 \|\cC^{1/2}\nabla\Psi (q_*)\|^2 - \|\cC^{1/2}\nabla\Psi (q)\|^2 \right)\right)- h(\mathcal{F}_{\Psi})\rv^p\\
&\les \EE_x \lv \delta^2\left( \|\cC^{1/2}\nabla\Psi (q_*)\|^2 - \|\cC^{1/2}\nabla\Psi (q)\|^2\right) \rv^p\\
& \les \delta^{2p}\EE_x \left[\nor{ q_* - q}^{2p}+ \nor{q}^p \nor{ q_* - q}^{p} \right] \les \delta^{3p}
\left( 1+\nor{q}^{2p}+ \nor{v}^{2p} \right),
\end{align*}
having used  the elementary inequality $\|a\|^2- \|b\|^2 \leq \|a-b\|^2+ 2 \|b\| \,\|a-b\|$ ,   \eqref{lemma6.1v}  and  
\eqref{lemma6.2iii}. 
\eqref{alpha2bit}  follows  after observing that $\lv h(y) - \bar{h}(y)\rv \leq \frac{y^2}{2}$. \eqref{alpha3bit} is a consequence of 
$\lv \bar{\alpha}\de-1\rv= \lv \bar{h}(\mathcal{F}_{\Psi})-\bar{h}(0) \rv$, together with the Lipshitzianity of $\bar{h}$. We are left with showing
\eqref{boundF}. To this end notice first that from \eqref{eqn:q*v*}
\begin{align}\label{use22}
\EE_x \nor{q_*-q -\delta v'}^p & =\EE_x \nor{(\cos\delta-1) q + (\sin \delta - \delta) v' - \frac{\delta}{2}\sin \delta \cC \nabla
\Psi(q)}^p \nonumber\\
& \les \delta^{2p}\left( 1+\nor{q}^{p}+ \nor{v}^{p}\right),
\end{align}
having used \eqref{lemma6.2ii} and \eqref{lemma6.1i}.
Analogously, from the definition of $v_*$ \eqref{eqn:q*v*}, 
\begin{align}\label{use2}
\EE_x \nor{v_*- v'}^p\les  \delta^{p}\left( 1+\nor{q}^{p}+ \nor{v}^{p}\right).
\end{align}
Therefore
\begin{align*}
\EE_x\lv \mathcal{F}_{\Psi}\rv^{p}&= \EE_x \lv \Psi(q)-\Psi(q_*) -\frac{\delta}{2}
\left( \langle \nabla \Psi(q), v' \rangle + \langle \nabla \Psi(q_*), v_* \rangle\right) \rv^p\\
& \leq \EE_x \lv \Psi(q_*)-\Psi(q)-\langle \nabla \Psi(q), (q_*-q) \rangle\rv^p
 + \EE_x \lv \langle \nabla \Psi(q), q_*-q -\delta v' \rangle  \rv^p\\
&+ \EE_x \lv\frac{\delta}{2} \langle \nabla \Psi(q),v'-v_*\rangle \rv^p
+ \EE_x \lv\frac{\delta}{2} \langle \nabla \Psi(q)- \nabla \Psi(q_*), v_*\rangle\rv^p\\
&  \les \EE_x \nor{q_* - q}^{2p}+ \|\nabla \Psi(q)\|_{-s}^p  \EE_x \nor{q_* - q-\delta v'}^{p}\\
& + \delta^p  \|\nabla \Psi(q)\|_{-s}^p \EE_x \nor{v_* -  v'}^{p}+ \delta^p \EE_x 
\left( \nor{q_* - q}^{p}\, \nor{v_* }^{p} \right)\les \delta^{2p}  \left( 1+\nor{q}^{2p}+ \nor{v}^{2p}\right),
\end{align*}
where in the second inequality we have used  \eqref{e.taylor.order2} and \eqref{lemma6.1iii}, in the third \eqref{der-s} together with \eqref{use22}, \eqref{use2},    Lemma \ref{psiq-psiq*} and Lemma \ref{lemma:q-q*}. 
This concludes the proof.
\end{proof}
Before starting the proof of Lemma \ref{lemma:q-q*}, recall the notation \eqref{v'}-\eqref{eqn:q*v*}. 
\begin{proof} [Proof of Lemma \ref{lemma:q-q*}]
$\xi\de$ is an $\cH^s$ valued Gaussian random variable with mean zero and covariance operator $\cC_s(I-e^{-2\delta\Gamma_2})$. Indeed $\Gamma_2$ is a bounded operator from $\cH^s$ into itself and  $B_2$ is an $\cH^s$ valued $\cC_s$-Brownian motion hence $\cC_s(I-e^{-2\delta\Gamma_2})$ is the product of a trace class operator times a bounded operator and therefore it is a trace class operator itself. So
$$
\EE\nor{\xi\de}^p\lesssim \left(\EE\nor{\xi\de}^2 \right)^{p/2}= \left[\tr(\cC_s (I-e^{-2\delta\Gamma_2}))\right]^{p/2}
\leq \left( \tr(\cC_s) \|I-e^{-2\delta\Gamma_2}\|_{\mathcal{L}(\h^s,\h^s)}\right)^{p/2}\lesssim \delta^{p/2}.  
$$
 This proves \eqref{lemma6.2i}. 
\eqref{lemma6.2ii} is a simple consequence of \eqref{lemma6.2i} together with \eqref{v'}. \eqref{lemma6.2iii} follows from  \eqref{eqn:q*v*}, using  \eqref{lemma6.1i} and \eqref{lemma6.2ii}:
\begin{align*}
\EE [\|q_*\kpod-q\kd\|_s^p\vert x\kd=x] &\les
\EE_x\nor{ q (\cos\delta -1)}^p\\
&+\EE_x \nor{ \sin\delta v'}^p+ \EE_x \nor{
\delta\sin\delta \,\cC \nabla \Psi(q)}^p\\
& \lesssim \delta^{2p}\nor{q}^p
+\delta^{p}\nor{v'}^p\lesssim\delta^{p}  (1+ \nor{q}^p+\nor{v}^p).
\end{align*}
\end{proof}

We now turn to prove Lemma  \ref{lemman1}. To this end we follow \cite[Proof of Lemma 3.7]{PST13}, highlighting the slight modifications needed in our context.
\begin{proof}[Proof of Lemma \ref{lemman1}]
In \cite[Theorem 5.1]{Berg:86} it is shown that proving the weak convergence of $\tilde{B}\de$ to $B$ boils down to showing that the following three limits hold in probability:
\begin{align}
&  \label{limits1}   \lim_{\delta\rightarrow 0}\delta 
\sum_{k\delta < t} \EE \left[ \nors{M\kd}^2 \vert \mathcal{F}\kd \right]  = t\,
 \mbox{Trace}_{\mathcal{H}^s\times \mathcal{H}^s}(\mathfrak{C}_s);\\
&   \lim_{\delta\rightarrow 0}\delta 
\sum_{k\delta < t} \EE \left[ \langle M\kd, \hat{\varphi}_j^{\ell} \rangle_{s\times s} \langle M\kd, \hat{\varphi}_i^{\bar{\ell}} \rangle_{s\times s} \vert \mathcal{F}\kd
\right]   =  t\, \langle \hat{\varphi}_j^{\ell}, \mathfrak{C}_s \hat{\varphi}_i^{\bar{\ell}} \rangle_{s\times s},
\quad \forall i,j \in \mathbb{N}, \, \ell,\bar{\ell}  \in  \{ 1,2 \};  \label{limits2} \\
&\lim_{\delta\rightarrow 0}\delta
\sum_{k\delta < t} \EE \left[  \nors{M\kd}^2  \mathbf{1}_{ \{ \nor{M\kd}^2 \geq \delta^{-1} \zeta \}} \vert \mathcal{F}\kd  \right]=0.
\qquad \forall \zeta >0, \label{limits3}
\end{align}
Here $M\kd$ and  $D\de(x)$ have been defined  in \eqref{eqn:Mdelta} and  \eqref{crop}, respectively, and $\mathcal{F}\kd$ is the filtration generated by $\{ x^{j,\delta}, \gamma^{j,\delta}, \xi\de,j=0, \dots, k \}$. 
\begin{description}
\item[Limit \eqref{limits1}] condition \eqref{lemmanoise2gen} implies that 
$$
\EE \left[ \nors{M\kd}^2 \vert \mathcal{F}\kd \right]=\tr_{\mathcal{H}^s\times \mathcal{H}^s} \mathfrak{C}_s+{\bf e}_1(x\kd)
$$
where $\lv {\bf e}_1(x\kd) \rv\les \delta^{b_2}(1+\nors{x\kd}^{d_2})$. Therefore
$$
\delta \sum_{k\delta< t}\EE \left[ \nors{M\kd}^2 \vert \mathcal{F}\kd \right]=t \tr_{\mathcal{H}^s\times \mathcal{H}^s} \mathfrak{C}_s+
\delta  \sum_{k\delta< t}\EE   [{\bf e}_1(x\kd)] .
$$
Thanks to \eqref{dde22}, we have
\be
\delta  \sum_{k\delta< t}\EE   \lv {\bf e}_1(x\kd) \rv & \les \delta^{b_2+1} 
\sum_{k\delta< t}(1+\EE\nors{x\kd}^{d_1})\\
& \leq \delta^{b_2} +\delta^{b_2+1-\eta_2} \left( \delta^{\eta_2} \sum_{k\delta< t}\EE\nors{x\kd}^{d_2} \right)
\stackrel{\delta\rightarrow 0}\longrightarrow 0.
\ee
\item[Limit \eqref{limits2}] can be proved as a consequence of \eqref{lemmanoise1gen} and \eqref{dde22}, acting as we did to show \eqref{limits1}.

\item[Limit \eqref{limits3}] the Cauchy-Schwartz and Markov inequalities give
\be
\EE \left[  \nors{M\kd}^2  \mathbf{1}_{ \{ \nor{M\kd}^2 \geq \delta^{-1} \zeta \}} \vert \mathcal{F}\kd  \right] & \leq \left(
 \EE[\nors{M\kd}^4 \vert \mathcal{F}\kd]   \right)^{1/2}   \left(
 \mathbb{P}[\nors{M\kd}^2 > \delta^{-1}\zeta] \right)^{1/2}\\
& \leq \frac{\delta}{\zeta} \EE[\nors{M\kd}^4 \vert \mathcal{F}\kd],
\ee
hence we need to estimate $ \EE[\nors{M\kd}^4 \vert \mathcal{F}\kd] $. To this end, we use \eqref{nonldrift} (which, we recall, is a consequence of 
\eqref{dde-G}) and \eqref{sizeincrem}:
\be
\EE[\nors{M\kd}^4 \vert \mathcal{F}\kd] & =\frac{1}{\delta^2} \EE[\nors{x\kpod-x\kd -\delta G\de(x\kd )}^4\vert  \mathcal{F}\kd]\\
& \les \frac{1}{\delta^2}\EE [\nors{x\kpod- x\kd}^4\vert \mathcal{F}\kd] + \delta^2 \EE
\nor{G\de(x\kd )}^4\\
& \les \delta^{4r-2}\EE(1+\nors{x\kd}^{4n})+\delta^2 \EE(1+\nors{x\kd }^{4p}).
\ee
Therefore
\be
&\delta \sum_{k\delta <t}\EE \left[  \nors{M\kd}^2  \mathbf{1}_{ \{ \nor{M\kd}^2 \geq \delta^{-1} \zeta \}} \vert \mathcal{F}\kd  \right] 
\les \delta^2 \sum_{k\delta<t}\EE \left[  \nors{M\kd}^4 \vert\mathcal{F}\kd \right] \\
&\les \delta^{4r} \sum_{k\delta <t}\EE(1+\nors{x\kd}^{4n}) +\delta^4 \sum_{k\delta <t}\EE(1+\nors{x\kd }^{4p})\\
&\les \delta^{4r-1}+\delta^{4r-\eta_3}\left( \delta^{\eta_3} \sum_{k\delta <t}\EE  \nors{x\kd}^{4n}\right)
+\delta^3+\delta^{4-\eta_4}\left( \delta^{\eta_4} \sum_{k\delta <t}\EE\nors{x\kd }^{4p}\right)\rightarrow 0.
\ee
\end{description}
\end{proof}

\begin{proof}[Proof of Lemma \ref{lemman2}] We will show the following two bounds:
  \begin{align}
& \lv \langle \hat{\varphi}_j^{\ell}, D\de(x)\, \hat{\varphi}_i^{\bar{\ell}} \rangle_{s\times s} - 
 \langle \hat{\varphi}_j^{\ell}, \mathfrak{C}_s\, \hat{\varphi}_i^{\bar{\ell}} \rangle_{s\times s} \rv \les \delta^{1/6} (1+ \nors{x}^{10})
\qquad \forall i,j \in \mathbb{N} \,\,\mbox{ and } \,\, \ell,\bar{\ell}\in {1,2};
 \label{lemmanoise1} \\
&  \lv       \textup{Trace}_{\mathcal{H}^s\times \mathcal{H}^s}(D\de (x))- 
 \textup{Trace}_{\mathcal{H}^s\times \mathcal{H}^s}(\mathfrak{C}_s)      \rv \les \delta^{1/6} (1+ \nors{x}^{10}) .                                       \label{lemmanoise2} 
\end{align}

Denote $\gamma_2:=\|\Gamma_2\|_{\mathcal{L}(\mathcal{H}^s)}$,  $M_1\kd:=\mathcal{P}_q(M\kd)$, 
$M_2\kd:=\mathcal{P}_v(M\kd)$ and recall from \eqref{mdeltax} that $M\de(x)=[M\kd\vert x\kd=x]$. Then, from 
 \eqref{eqn:Mdelta}, 
\begin{align}
M_1\de(x) &=\frac{1}{\sqrt{2\delta}}\left[q^{1,\delta}-q-\EE[q^{1,\delta}-q]\right]\label{decomp1},\\
M_2\de(x) &=\frac{1}{\sqrt{2\delta\Gamma_2}}\left[ v^{1,\delta}-v-\EE[v^{1,\delta}-v] \right] \label{decomp2}.
\end{align}
In order to obtain \eqref{lemmanoise1} and \eqref{lemmanoise2}, we start with studying
$M_1\de(x)$ and $M_2\de(x)$. More precisely, we proceed as follows:
\begin{itemize}
\item We first show the bound
\begin{align}\label{nec1}
\EE\left[  \nor{M\kd_1}^2 \vert x\kd = x \right]=\EE\nor{M\de_1(x)}^2 \lesssim \delta (1+\nors{x}^2)
\end{align}
and the decomposition
\begin{align}\label{M2Rxi}
M\de_2(x)=\frac{1}{\sqrt{2\delta\Gamma_2}}\left[ R\de(x)+\xi\de\right], 
\end{align}
where $R\de(x)$, defined in \eqref{Rdelta}, is such that
\begin{align}\label{p1}
\frac{1}{\delta}\,\EE \nor{R\de(x)}^2 \,\les \delta^{1/3}(1+\nors{x}^{10}).
\end{align}
\item We then prove that \eqref{lemmanoise1} and \eqref{lemmanoise2} are a consequence of \eqref{nec1} and \eqref{M2Rxi}-\eqref{p1},
together with 
\be 
\lv\, \frac{1}{2\delta\gamma_2}\EE\nor{\xi\de}^2 - \textup{Trace}_{\mathcal{H}^s} \cC_s \, \rv\les \delta^2, 
\ee
which is easily seen to hold true. Indeed by definition
\begin{align}
\lv\, \frac{1}{2\delta\gamma_2}\EE\nor{\xi\de}^2 - \textup{Trace}_{\mathcal{H}^s} \cC_s \, \rv
&= \lv\, \textup{Trace}_{\mathcal{H}^s} \left[\frac{\cC_s - \cC_s e^{-2\delta\Gamma_2}}{2\delta\gamma_2} - \cC_s \right]\rv \nonumber \\
& \leq \textup{Trace}_{\mathcal{H}^s}(\cC_s)\lv
\frac{1-e^{-2\delta\gamma_2}}{2\delta\gamma_2}-1 \rv\les \delta^2.\label{blublu}
\end{align}
\end{itemize}
 \eqref{nec1} is a straightforward  consequence of \eqref{decomp1}, using \eqref{blabla}, \eqref{fact1} and
\eqref{lemma6.2iii}:
\be
\EE\left[  \nor{M\kd_1}^2 \vert x\kd = x \right] \les \frac{1}{\delta} \EE\nor{q_* -q}^2 \les \delta (1+\nor{x}^2).
\ee 
Recalling that  $\gamma\de\sim \textup{Bernoulli}(\alpha\de)$ (with $\alpha\de$ defined by equation 
\eqref{alphade}), to decompose $M_2\de(x)$ we start from \eqref{decomp2} and use \eqref{blabla}:
\begin{align*}
M_2\de(x)=\frac{1}{\sqrt{2\delta \Gamma_2}}\left[  \gamma\de(v_*+v')-\EE(\gamma\de(v_*+v'))
-v'-v+\EE(v'+v) \right].
\end{align*}
By \eqref{v'} and \eqref{fact2}, 
\begin{align}\label{star}
-v'-v+\EE(v'+v)= -v'+\EE v'  =-\xi\de;
\end{align}
  so \eqref{eqn:q*v*} yields:
\begin{align}
\sqrt{2\delta \Gamma_2}M_2\de (x)&= (\gamma\de-\EE(\gamma\de)) \left[-q\sin\delta -\frac{\delta}{2}\cos\delta\,\, \cC\nabla\Psi(q)\right] -\xi\de \nonumber \\
&+ \gamma\de\left[ (\cos\delta+1)v'-\frac{\delta}{2}\cC \nabla\Psi(q_*)\right]
-\EE\left[  \gamma\de\left( (\cos\delta+1)v'-\frac{\delta}{2}\cC \nabla\Psi(q_*)\right)  \right] \nonumber.
\end{align}
 Let $f(x)=f\de(x)+\bar{f}(x)$, with
\be
 f\de(x):=-\frac{\delta}{2}\cC \nabla\Psi(q_*) \quad \mbox{and} \quad \bar{f}(x):=(\cos\delta +1 )v'.
\ee
Then
\begin{align*}
 \sqrt{2\delta \Gamma_2}M_2\de(x) &=(\gamma\de-\EE(\gamma\de)) (-q\sin\delta -\frac{\delta}{2}\cos\delta\, \cC\nabla\Psi(q))\\
&+\gamma\de f-\EE(\gamma\de f)-\xi\de\\
&=(\gamma\de-\EE(\gamma\de)) (-q\sin\delta -\frac{\delta}{2}\cos\delta \,\cC\nabla\Psi(q))\\
&+\gamma\de f-\EE(\gamma\de) f + \left[\EE(\gamma\de) -1\right] f-
\EE\left[(\gamma\de-1) f\right]+f-\EE f-\xi\de.
\end{align*}
However $\bar{f}-\EE \bar{f}=(\cos\delta+1)(v'-\EE v')$, so using \eqref{star}
$$
f-\EE f-\xi\de= \bar{f}-\EE \bar{f}+f\de- \EE f\de-\xi\de=(\cos\delta)\xi\de+f\de- \EE f\de;
$$
therefore, setting
\begin{align}
& R\de_1(x):=(\EE(\gamma\de)-\gamma\de) (q\sin\delta  +\frac{\delta}{2}\cos\delta \,\cC\nabla\Psi(q))  \nonumber \\
& R\de_2(x):=(\gamma\de-\EE\gamma\de) f= (\gamma\de-\EE\gamma\de)  [(\cos\delta+1)v' - \frac{\delta}{2}
 \cC\nabla\Psi(q_*) ]   \nonumber \\
& R_3\de(x):=- \EE[(\gamma\de -1)f]+[(\EE\gamma\de)-1]f     \nonumber  \\
& R_4\de(x):=f\de-\EE(f\de)+(\cos\delta -1)\xi\de \qquad \mbox{and}   \nonumber  \\
&R\de(x):=\sum_{i=1}^4R_i\de \label{Rdelta}, 
\end{align}
we obtain \eqref{M2Rxi}.
From now on, to streamline the notation, we will not keep track of the $x$-dependence in $R_i\de(x)$ and $R\de(x)$. In other words, we will simply denote $R_i\de:=R_i\de(x)$ and $R\de:=R\de(x)$.

To prove \eqref{p1} we bound  $\EE\nor{R_i\de}^2$,   $i=1,\dots,4$. Observe first that
$$
\EE(\gamma\de -\EE \gamma\de)^2= \EE(\alpha\de) (1-\EE \alpha\de)\les \delta^2 (1+\nors{x}^4), 
$$
which is a consequence of  \eqref{alpha-1} and \eqref{fact1}. Therefore, by \eqref{lemma6.1i},
\begin{equation}\label{estR1}
\EE \nor{R\de_1}^2\les \EE(\gamma\de -\EE \gamma\de)^2 \nor{\delta q+\delta \cC\nabla\Psi(q)}^2\les \delta^3(1+\nors{x}^6).
\end{equation}
Now notice that the Bochner's inequality together with Jensen's inequality give
$$
\nor{\EE g}^2\leq\EE \nor{g}^2, \qquad \mbox{for every $\mathcal{H}^s$-valued,  integrable } g.
$$ 
To bound $R_2\de$ we split it  into two terms, namely
$$
R_2\de:=(\gamma\de -\EE \gamma\de) (\cos\delta +1) v' - (\gamma\de -\EE \gamma\de) \frac{\delta}{2}\cC \nabla\psi(q_*)=:
R_{21}\de+ R_{22}\de.
$$
To estimate $R\de_{22}$ we use \eqref{fact1},  \eqref{lemma6.1i} and  \eqref{lemma6.2iii}:
$$
\EE \nor{R\de_{22}}^2 \les \delta^2 \left[\EE \nor{\cC \nabla\Psi (q_*)-\cC \nabla\Psi (q)  }^2 + \nor{\cC \nabla\Psi (q)}^2\right]
\les \delta^2 (1+\nor{q}^2).
$$
To study $R_{21}\de$ instead, we write $\gamma\de-\EE\gamma\de=\gamma\de-1+\EE(1-\gamma\de)$ and we repeatedly use 
\eqref{ops}, obtaining:
\be
\EE\nor{R_{21}\de}^2  &\les \EE\left[ (\gamma\de -\EE \gamma\de)^2 \nor{v'}^2 \right]\\
& \les   \EE \left[ (\gamma\de -1)^2  \nor{v'}^2 \right]+ \EE(1-\gamma\de)^2 \EE\nor{v'}^2\\
& \stackrel{\eqref{lemma6.2ii}}{\les} \left( \EE\lv\gamma\de -1\rv^3  \right)^{2/3} \left( \EE \nor{v'}^6\right)^{1/3}+
\EE\lv  1-\alpha\de\rv (1+\nor{v}^2)\\
&\stackrel{\eqref{alpha-1}}{\les} \delta^{4/3}(1+\nors{x}^8) +\delta^2 (1+\nors{x}^6)\les \delta^{4/3}(1+\nors{x}^8) .
\ee
Combining the estimates of $R\de_{21}$ and $R\de_{22}$ we get
\begin{equation}\label{estR2}
\EE \nor{R\de_2}^2\les \delta^{4/3}(1+\nors{x}^8).
\end{equation}
As for $R_3\de$, using $\EE\nor{f}^2\les 1+\nors{x}^2$  (which is a consequence of \eqref{lemma6.1i} and 
\eqref{lemma6.2ii}),
\be
\EE\nor{R\de_3}^2&\leq \EE \nor{ (\EE(\gamma\de)-1) f }^2+\nor{ \EE[(\gamma\de-1)f] }^2\\
& \leq (\EE(\alpha\de)-1)^2 \EE\nor{f}^2 + \EE\nor{(\gamma\de-1)f }^2\\
&\stackrel{\eqref{alpha-1}}{\les} \delta^4 (1+\nors{x}^8)(1+\nors{x}^2) +  \left( \EE\lv\gamma\de-1\rv^3  \right)^{2/3}
 \left(  \EE \nor{f}^6  \right)^{1/3}\\
& \stackrel{\eqref{ops}}{\les} \delta^4 (1+\nors{x}^{10}) +\delta^{4/3} (1+\nors{x}^4)^{2/3} (1+\nors{x}^2) \leq  
\delta^{4/3} (1+\nors{x}^{10}).\label{estR3}
\ee
Now the last term: from \eqref{lemma6.1i},  $\EE\nor{f\de}^2 \les \delta^2(1+\nor{q}^2)$; therefore
\begin{align}\label{estR4}
\EE\nor{R_4\de}^2 \les \EE\nor{f\de}^2 +\delta^4 \EE\nor{\xi\de}^2\les \delta^2(1+\nor{q}^2).
\end{align}
It is now clear that \eqref{p1} follows from \eqref{estR1}, \eqref{estR2}, \eqref{estR3} and \eqref{estR4}.

 Let us now show that \eqref{lemmanoise1} and \eqref{lemmanoise2} follow from \eqref{nec1} and \eqref{M2Rxi}-\eqref{p1}.  We   start with  \eqref{lemmanoise2}.  
 By definition, 
\begin{align}\label{tr1}
\textup{Trace}_{\mathcal{H}^s\times \mathcal{H}^s} (D\de(x))&= \EE\left[  \nor{M\kd}^2 \vert x\kd = x \right]\nonumber\\
&= \EE\left[  \nor{M\kd_1}^2 \vert x\kd = x \right]+ \EE\left[  \nor{M\kd_2}^2 \vert x\kd = x \right]
\end{align}
and 
$$
\textup{Trace}_{\mathcal{H}^s\times \mathcal{H}^s}(\mathfrak{C}_s)=
\textup{Trace}_{\mathcal{H}^s}(\cC_s).
$$
Also, 
\begin{align}
 & \lv \EE\left[  \nor{M\kd_2}^2 \vert x\kd = x \right]  - \textup{Trace}_{\mathcal{H}^s}(\cC_s) \rv \stackrel{\eqref{M2Rxi}}{=}
\lv \frac{1}{2\delta \gamma_2} \EE\nor{R\de+\xi\de}^2-\textup{Trace}_{\mathcal{H}^s}(\cC_s)\rv\nonumber\\
&\leq \frac{1}{2\delta \gamma_2} \EE\nor{R\de}^2+ 
\lv  \frac{1}{2\delta \gamma_2}\EE\nor{\xi\de}^2  -\textup{Trace}_{\mathcal{H}^s}(\cC_s)\rv
+\frac{1}{\delta \gamma_2}\EE \langle R\de,\xi\de\rangle_s\nonumber\\
& \les  \frac{1}{2\delta \gamma_2} \EE\nor{R\de}^2+ 
\lv  \frac{1}{2\delta \gamma_2}\EE\nor{\xi\de}^2  -\textup{Trace}_{\mathcal{H}^s}(\cC_s)\rv
+\frac{1}{\delta}\left(\EE\nor{R\de}^2  \right)^{1/2} \left(\nor{\xi\de}^2  \right)^{1/2}\nonumber\\
&\stackrel{\eqref{blublu}}{\les} \delta^{1/3}(1+\nors{x}^{10})+\delta^2+\delta^{1/6}(1+\nors{x}^5)\les 
\delta^{1/6}(1+\nors{x}^{10}).\label{tr2}
\end{align}
\eqref{tr1} and the above \eqref{tr2} imply \eqref{lemmanoise2}. \eqref{lemmanoise1} can be obtained similarly.
Due to the symmetry of $D\de(x)$, all we need to show is that 
 if at least one index between $\ell$ and
$\bar{\ell}$ is equal to 1 then
\begin{align}\label{d1}
\langle \hat{\varphi}_i^{\ell}, D\de(x) \hat{\varphi}_j^{\bar{\ell}} \rangle_{s\times s} \les  \delta^{1/2}(1+\nors{x}^5).
\end{align}
If instead $\ell=\bar{\ell}=2$  we will prove that
\begin{align}\label{d2}
\lv \langle \phi_i, D\de_{22}(x) \varphi_j\rangle_s - \langle \varphi_i, \cC_s \varphi_j \rangle_s \rv\leq \delta^{1/6}
(1+\nors{x}^{10}),
\end{align}
where $D\de_{22}(x)=\EE [M\kd_2 \otimes M\kd_2 \vert x\kd=x]$. \eqref{d1} and \eqref{d2} imply \eqref{lemmanoise1}. 
To prove the bound \eqref{d1}  observe first that 
$$
\EE \nor{M_2\de(x)}^2\les 1+\nors{x}^{10},
$$
which follows from \eqref{M2Rxi}, \eqref{p1} and \eqref{lemma6.2i}.
To show \eqref{d1} suppose, without  loss of generality, that $\ell=1,\bar{\ell}=2$. Then
\be
\lv \langle \hat{\varphi}_i^{1}, D\de(x) \hat{\varphi}_j^{2}\rangle_{s\times s}\rv\leq 
\EE \lv\langle M\de(x), \hat{\varphi}_i^{1}\rangle_{s\times s} \langle M\de(x), \hat{\varphi}_i^{2}\rangle_{s\times s}  \rv\\
\leq \left(\EE \nor{M_1\de(x)}^2 \right)^{1/2} \left(\EE \nor{M_2\de(x)}^2 \right)^{1/2}\stackrel{\eqref{nec1}}{\les} \delta^{1/2}(1+\nors{x}^5).
\ee
As for \eqref{d2}, let $\xi:=\xi^0$, i.e. let $\xi$ be a mean zero Gaussian random variable with covariance operator 
$\cC_s$ in $\mathcal{H}^s$. Then 
\begin{align*}
&\lv \langle \hat{\varphi}_i^2, D\de(x) \hat{\varphi}_j^2\rangle_{s\times s}
-\langle   \hat{\varphi}_i^2, \mathfrak{C}_s \hat{\varphi}_j^2 \rangle_{s\times s}\rv
=  \lv \langle\varphi_i, D\de_{22}(x)\varphi_j \rangle_s - \langle\varphi_i,\cC_s \varphi_j \rangle_s\rv\\
&=\lv \EE \left(  \langle M_2\de(x),\varphi_i\rangle_s \langle M_2\de(x),\varphi_j\rangle_s \right) 
-\EE \left( \langle \xi,\varphi_i\rangle_s \langle \xi,\varphi_j\rangle_s \right)\rv\\
&\stackrel{\eqref{M2Rxi}}{=} 
\lv  \EE \langle \frac{R\de+\xi\de}{\sqrt{2\delta \Gamma_2}} , \varphi_i\rangle_s
 \langle \frac{R\de+\xi\de}{\sqrt{2\delta \Gamma_2}} , \varphi_j \rangle_s-
\EE \left( \langle \xi,\varphi_i\rangle_s \langle \xi,\varphi_j\rangle_s \right) \rv\\
&\les \frac{1}{\delta}\EE\nor{R\de}^2 + \frac{1}{\delta} \left(\EE\nor{R\de}^2 \right)^{1/2}
\left(\EE\nor{\xi\de}^2 \right)^{1/2}+ 
\lv  \frac{1}{2\delta \gamma_2}\EE (\langle  \xi\de,\varphi_i \rangle_s)^2-\EE (\langle \xi, \varphi_i \rangle_s) \rv;
\end{align*}
so, by using again \eqref{p1} and \eqref{lemma6.2i} and with a reasoning analogous to that contained in 
\eqref{blublu} and \eqref{tr2}, we obtain \eqref{d2}.
\end{proof}

\section*{Appendix C}
We gather here some basic facts about Hamiltonian mechanics. For a more thorough discussion the reader may consult \cite{sanz1994numerical, neal2010mcmc}. 

Let us start from the Hamiltonian formalism in a finite dimensional setting. 
To a given  real valued and smooth function $\Hf (q,p): \RR^{2N} \rightarrow \RR$, we can associate in a canonical way a system of differential equations, the so called {\em canonical Hamiltonian system associated to} $\Hf$, namely
\begin{align*}
\frac{dq}{dt}&= \nabla_p \Hf (q,p)\\
\frac{dp}{dt}&= -\nabla_q \Hf (q,p)\,.
\end{align*}
Using the symplectic matrix 
\be
J = \left( \begin{array}{cc}
0 & I  \\
-I & 0  \\
 \end{array} \right),
\ee
 and denoting $z=(q,p)\in \mathbb{R}^{2N}$, the canonical Hamiltonian system can be rewritten as
\be\label{hampr}
\frac{dz}{dt}= J \nabla_z \Hf (q,p)\,.
\ee
The two properties of the  Hamiltonian flow \eqref{hampr} that are relevant to our purposes are: i) smooth functions of the Hamiltonian $\Hf$ remain consant along the solutions of \eqref{hampr}; ii) the flow  preserves the volume element $dz$. As a consequence, the Hamiltonian dynamics preserves any measure with density $e^{-\Hf (z)}$ with respect to Lebesgue measure. Clearly, an analogous discussion holds for any systems obtained by making the non-canonical choice
$$
\hat{J} = \left( \begin{array}{cc}
0 & L  \\
-L & 0  \\
 \end{array} \right), \quad    L\,  \mbox{ any symmetric matrix,}
$$
 with corresponding dynamics $$ \frac{dz}{dt}= \hat{J} \nabla_z \Hf (q,p)\,.$$ 

This reasoning can be repeated in our infinite dimensional context (however in this case one cannot talk about conservation of volume element $dz$). The Hamiltonian part of the equations considered in Section  \ref{subs:velnomom} is built precisely in this spirit. The change of variable which allows us to swap from momentum to velocity variable corresponds to going from  the canonical to the non canonical choice.    In particular, once we fix $\mathcal{M}= \mathcal{L}= \C^{-1}$ , our non-canonical symplectic matrix is 
$$
\hat{J} = \left( \begin{array}{cc}
0 & \C \\
-\C & 0  \\
 \end{array} \right).
$$
For more comments about the particular form of the Hamiltonian function in our infinite dimensional setting see Remark \ref{rem:wellposaccprob} and  \cite{BPSS11}.

\bibliographystyle{alpha}
\bibliography{difflim}

\end{document}